\documentclass{amsart}

\usepackage{ifthen}
\usepackage{mathsymbols}

\usepackage[all]{xy}
\usepackage{url}

\newtheorem{anyprop}{Anyprop}[section]

\newtheorem{theorem}[anyprop]{Theorem}
\newtheorem{lemma}[anyprop]{Lemma}
\newtheorem{proposition}[anyprop]{Proposition}
\newtheorem{corollary}[anyprop]{Corollary}

\theoremstyle{definition}
\newtheorem{algorithm}[anyprop]{Algorithm}

\newtheorem{example}[anyprop]{Example}

\newtheorem{problem}[anyprop]{Problem}

\newtheorem{question}[anyprop]{Question}

\newtheorem{remark}[anyprop]{Remark}

\newtheorem*{acknowledgement}{Acknowledgement}

\theoremstyle{remark}

\numberwithin{equation}{section}

\usepackage{amscd}
\usepackage{amssymb}

\begin{document}

\title[Semistable Vector Bundles and Tannaka duality ]
{Semistable vector bundles and Tannaka duality from a computational point of view}

\author[Almar Kaid and Ralf Kasprowitz]{Almar Kaid and Ralf Kasprowitz}

\address{Parkstieg 6, 22143 Hamburg, Germany}

\email{akaid@uni-osnabrueck.de}

\address{Universit\"at Paderborn, Fakult\"at f\"ur Elektrotechnik, Informatik und Mathematik,
Institut f\"ur Mathematik, Warburger Str. 100, 33098 Paderborn, Germany}

\email{kasprowi@math.upb.de}





\begin{abstract}
We develop a semistability algorithm for vector bundles which are given as a kernel of a surjective morphism between splitting bundles on the projective space $\PP^N$ over an algebraically closed field $K$. This class of bundles is a generalization of syzygy bundles. We show how to implement this algorithm in a computer algebra system. Further we give applications, mainly concerning the computation of Tannaka dual groups of stable vector bundles of degree $0$ on $\PP^N$ and on certain smooth complete intersection curves. We also use our algorithm to close an open case left in a recent work of 
L. Costa, P. Macias Marques and R. M. Mir\'o-Roig regarding the stability of the syzygy bundle of general forms. Finally, we apply our algorithm 
to provide a computational approach to tight closure. All algorithms are implemented in the computer algebra system CoCoA.
\end{abstract}

\maketitle

Mathematical Subject Classification (2010): primary: 14J60, 14Q15; \\
secondary: 13P10

Keywords: Semistable vector bundle, syzygy bundle, Tannaka duality, monodromy group, tight closure

\section{Introduction} \label{introduction}

The notion of slope-(semi)stability for vector bundles on a smooth projective varieties over an algebraically closed field $K$, as introduced by D. Mumford in the case of curves and generalized by F. Takemoto to higher dimensional varieties, is a very important tool in algebraic geometry. Unfortunately, for a concretely given vector bundle it is often very difficult to decide whether it is semistable or even stable. In this paper we develop an algorithm to determine computationally the semistability of certain vector bundles on the projective space $\PP^N$. Throughout this paper we assume that $N \geq 2$, since for $N=1$ by the Theorem of A. Grothendieck every vector bundle splits as a direct sum of line bundles. We restrict ourselves  to vector bundles which are given as a kernel of a surjective morphism between splitting bundles, i.e., vector bundles $\shE$ which sit in a short exact sequence $$0 \lto \shE \lto \bigoplus_{i=1}^n \mathcal{O}_{\PP^N}(a_i) \stackrel{\varphi}\lto \bigoplus_{j=1}^m\mathcal{O}_{\PP^N}(b_j) \lto 0.$$ We call such bundles \emph{kernel bundles}. For instance, by the theorem of Horrocks every non-split vector bundle on $\PP^2$ admits such a presentation. The morphism $\varphi$ which defines $\shE$ is given by an $m \times n$ matrix $\shM=(a_{ji})$, where the entries $a_{ji} \in R:=K[X_0 \komdots X_N]$ are homogeneous polynomials of degrees $b_j - a_j$. Special instances ($m=1$ and $b_1 = 0$) of kernel bundles are the so-called \emph{syzygy bundles} $\Syz(f_1 \komdots f_n)$ for $R_+$-primary homogeneous polynomials $f_1 \komdots f_n$ (i.e., $\sqrt{(f_1 \komdots f_n)} = R_+$), i.e., a syzygy bundle has a presenting sequence 
$$0 \lto \Syz(f_1 \komdots f_n) \lto \bigoplus_{i=1}^n \mathcal{O}_{\PP^N}(-d_i) \stackrel{f_1 \komdots f_n}\lto \mathcal{O}_{\PP^N} \lto 0,$$
where $d_i = \deg(f_i)$. Due to their explicit nature, kernel bundles and syzygy bundles are suitable for direct computations, in particular 
Gr\"obner basis methods and combinatorics. But still in general, not much is known about (semi)stability of kernel bundles or even syzygy bundles. One of the most important results in this direction, due to H. Brenner, is a combinatorial criterion for (semi)stability of syzygy bundles given by monomial families:

\begin{theorem}[Brenner] \label{monomialcase}
Let $K$ be a field, $R:=K[X_0 \komdots X_N]$ and let $f_i=X^{\sigma_i}$ denote $R_+$-primary monomials of degrees $d_i=|\sigma_i|$ in $K[X_0 \komdots X_N]$, $i=1 \komdots n$. Suppose that for every subset $J \subseteq I := \{1 \komdots n\}$, $|J| \geq 2$, the inequality $$\frac{d_J - \sum_{i \in J} d_i}{|J|-1} \leq \frac{-\sum_{i \in I} d_i}{n-1}$$ holds, where $d_J$ is the degree of the highest common factor of $f_i$, $i \in J$. Then the syzygy bundle $\Syzfn$ is semistable $($and stable if $<$ holds$)$. \end{theorem}

\begin{proof}
See \cite[Corollary 6.4]{brennerlookingstable}.
\end{proof}

Another important theorem, due to G. Bohnhorst and H. Spindler, is a numerical (semi)stability criterion for kernel bundles of rank $N$ on $\PP^N$ in characteristic $0$:

\begin{theorem}[Bohnhorst-Spindler]
\label{bohnhorstspindlerhdoneintro}
Let $\shE$ be a vector bundle of rank $N \geq 2$ on the projective space $\PP^N$ over an algebraically closed field $K$ of characteristic~$0$.
Suppose there is a short exact sequence
$$0 \lto \shE \lto \bigoplus_{i=1}^{N+k} \mathcal{O}_{\PP^N}(a_i) \lto \bigoplus_{j=1}^k \mathcal{O}_{\PP^N}(b_j) \lto 0,$$
such that $a_1 \geq \ldots \geq a_{N+k}$, $b_1 \geq \ldots \geq b_k$ and $b_j > a_{j}$ for $j=1 \komdots k$. Then
$\shE$ is semistable $($stable$)$ if and only if $a_{N+k} \geq (>) \mu(\shE) = \frac{1}{N}(\sum_{i=1}^{N+k} a_i - \sum_{j=1}^{k}b_j)$.
\end{theorem}

\begin{proof}
This is \cite[Theorem 2.7]{bohnhorstspindler} applied to the dual bundle $\shE^*$.
\end{proof}

A general algorithm using Gr\"obner bases methods (computation of syzygy modules) that detects semistability of syzygy bundles and its imlementation by the first author was already announced by H. Brenner in \cite[Remark 5.3]{brennerlookingstable}. In this article we describe this algorithm more generally for kernel bundles and describe in detail how to implement it in a computer algebra system (this has been done concretely by the first author in CoCoA \cite{CocoaSystem}). This semistability algorithm can be used as a tool to examine further problems regarding semi(stability) of vector bundles by providing interesting examples. We explain these applications in more detail in the sequel. The paper is organized as follows. \medskip

In Section \ref{hoppesemistabilitycriterion} we recall a criterion due to H. J. Hoppe (see Lemma \ref{exteriorpowercriterion}) which relates (semi)stability to global sections of exterior powers of a given vector bundle. In particular, we show that this result, originally only formulated in characteristic $0$, holds in arbitrary characteristic. Hoppe's criterion is the key-result for our algorithm. \medskip

In Section \ref{syzygyandkernelbundles} we discuss some properties of kernel bundles and syzygy bundles on projective spaces. In particular, for these bundles we discuss necessary Bohnhorst/Spindler like numerical conditions (compare Theorem \ref{bohnhorstspindlerhdoneintro}) for semistability. \medskip 

The actual semistability algorithm for kernel bundles and its implementation is explained in Section \ref{algoimplementation}. Besides exterior powers, we also describe explicitly how to compute global sections of tensor products and symmetric powers of kernel bundles. These algorithms play an important role in our first application: the computation of Tannaka dual groups of polystable vector bundles $\shE$ of degree $0$ and rank $r$ on $\PP^N$ in characteristic $0$. \medskip 

Section \ref{tannaka} starts with a brief introduction to Tannaka duality. Roughly spoken, for a polystable vector bundle $\shE$ of degree $0$ one can find a semisimple algebraic group $G_{\shE}$ and an equivalence of categories between the abelian tensor category generated by $\shE$ and the category of finite-dimensional representations of $G_{\shE}$. The algebraic group $G_{\shE}$ is called the \emph{Tannaka dual group} of $\shE$. It was shown by the second author in \cite[Lemma 4.4 and Proposition 5.3]{monodromygroups} that for stable vector bundles of degree $0$ as in Theorem \ref{bohnhorstspindlerhdoneintro} the almost simple components of the Tannaka dual group are of type $A$. We explain how to compute the Tannaka dual group for an arbitrary stable kernel bundle of degree $0$ on $\PP^N$ and construct examples for low-rank syzygy bundles on $\PP^2$ having the symplectic group $\mathbf{Sp}_{r}$ as Tannaka dual group. \medskip

Furthermore, we are interested in the behaviour of the Tannaka dual group after restricting the bundles to smooth curves. Section \ref{sectionrestrictiontheorems} contains a short overview of restriction theorems for sheaves. 

In Section \ref{tannakaoncurves}, we describe a method to construct for certain kernel bundles $\shE$ on $\PP^N$ a finite morphism $f: \PP^N \to \PP^N$ such that the restriction of the pull-back $f^*(\shE)$ to certain complete intersection curves of sufficiently large degree has the same Tannaka dual group as the vector bundle $f^*(\shE)$ on $\PP^N$. We show that this works for example for the syzygy bundles constructed in Section 5. We would like to draw the reader's attention to the paper \cite{douhcompact} by V. Balaji. He shows the existence of a rank $2$ bundle $\shE$ with $c_2(\shE) \gg 0$ on a smooth surface $X$, such that the restriction to a curve of genus $>1$ has Tannaka dual group $\mathbf{SL}_2$, see also \cite[Proposition 3]{adddouhcompact}. His method is completely different from ours. He uses this result to show that the moduli space of stable principal $H$-bundles on $X$ with large characteristic classes is non-empty, where $H$ is any semisimple algebraic group (\cite[Chapter 7]{douhcompact}). \medskip 

In Section \ref{stabilityofgenericquadrics} we close an open case left in the paper \cite{costamiroroigmarques}, where L. Costa, R. M. Mir\'o-Roig and P. Macias Marques show the stability of the generic syzygy bundle on $\PP^2$ except for the bundle generated by five generic quadrics. We use the results obtained in Section 5 and construct an example for a stable syzygy bundle in this case, which gives the generic result via the openness of stability. \smallskip

In the final section we provide another application of the semistability algorithm concerning the computation of tight/solid closure of homogeneous ideals in the coordinate ring of a smooth projective curve. This is possible due to the geometric approach to this topic developed by H. Brenner.

\begin{acknowledgement}
We would like to thank Holger Brenner for many useful discussions. In particular the first author is grateful for the supervision of his PhD-thesis \cite{kaiddissertation} at the University of Sheffield where the semistability algorithm is part of one chapter. 
\end{acknowledgement}

\section{Theoretical backround -- Hoppe's semistability criterion} \label{hoppesemistabilitycriterion}

We recall that a  torsion-free sheaf $\shE$ on a smooth projective variety $X$ over an algebraically closed field $K$ is \emph{semistable} if for every coherent subsheaf $0 \neq \shF \subset \shE$ the inequality $\mu(\shF) :=\deg(\shF)/\rk(\shF) \leq \deg(\shE)/\rk(\shE) = \mu(E)$ holds. The sheaf $\shE$ is \emph{stable} if the inequality is always strict. The degree of a sheaf $\shF$ is defined using intersection theory and a fixed very ample invertible sheaf $\mathcal{O}_X(1)$ (which is also called a \emph{polarization} of $X$) as $\deg(F) = \deg(c_1(\shF). \mathcal{O}_X(1)^{\dim(X)-1})$. For every coherent torsion-free sheaf $\shE$ there exists a unique filtration $\shE_1 \subset \shE_2 \subset \ldots \shE_t = \shE$, called the \emph{Harder-Narasimhan filtration}, such that $\shE_i/\shE_{i-1}$ is semistable and $\mu(\shE_1) > \mu(\shE_2/\shE_1) > \ldots > \mu(\shE/\shE_{t-1})$. The slopes $\mu(\shE_1)$ and $\mu(\shE/\shE_{t-1}$) are also denoted by $\mu_{\max}(\shE)$ and $\mu_{\min}(\shE)$ respectively. If $K$ is not algebraically closed, then we define the terms degree, semistable, etc. via the algebraic closure of $K$.


%

If the characteristic of the base field $K$ is $0$, it is well-known that the tensor product $\shE \otimes \shF$ of two
semistable vector bundles $\shE$ and $\shF$ on a smooth projective polarized variety $(X,\mathcal{O}_X(1))$ is again semistable, and this also holds for exterior powers and symmetric powers (cf. \cite[Theorem 3.1.4 and Corollary 3.2.10]{huybrechtslehn}). This does not longer hold in characteristic $p>0$. This is due to the fact that the (absolute) \emph{Frobenius morphism} $F: X \ra X$ may destroy semistability, i.e., the Frobenius pull-back $F^*(\shE)$ of a semistable vector bundle $\shE$ is in general not semistable; see for instance the example of Serre in \cite[Example 3.2]{hartshorneamplecurve}. But for vector bundles on a projective space (in which we are mainly interested in) semistability behaves nicely with respect to tensor operations.

\begin{lemma}
\label{semistabilitytensorposchar}
Let $(X, \mathcal{O}_X(1))$ be a smooth projective polarized variety defined over an algebraically closed field $K$ of positive characteristic such that $\mu_{\max}(\Omega_X) \leq  0$. If $\shE$ and $\shF$ are semistable vector bundles, then $\shE \otimes \shF$, all exterior powers
$\bigwedge^q \shE$ and all symmetric powers $S^q \shE$ are semistable. In particular, this holds if $X$ is an abelian variety, a toric variety or a homogeneous space.
\end{lemma}

\begin{proof}
Since $\mu_{\max}(\Omega_X) \leq  0$ a semistable vector bundle $\shE$ on $X$ is \emph{strongly semistable} by \cite[Theorem 2.1]{mehtaramanathanhomogeneous} (we recall that this means that the Frobenius pull-backs $\efpb(\shE)$ are semistable for all $e \geq 0$). Hence it follows from \cite[Theorem 3.23]{ramananramanathan} that $\shE \otimes \shF$, $\bigwedge^q \shE$ and $S^q \shE$ are also semistable.

The cotangent bundle of an abelian variety is trivial and the cotangent bundles of toric varieties and homogeneous spaces can be
embedded into a trivial bundle. So these varieties fulfill the condition $\mu_{\max}(\Omega_X) \leq  0$ which gives the supplement.
\end{proof}

The following result is well-known in characteristic $0$ (see for instance \cite[Proposition 2.1.5]{brennerlookingstable} or \cite[Proposition 1.1]{bohnhorstspindler}). It gives an algorithmic criterion to check semistability of a vector bundle on $\PP^N$ in terms of global sections of its exterior powers. It uses the trivial but useful fact (in particular for a computational approach to semistability) that a semistable vector bundle of
negative degree (or slope) does not have any nontrivial global sections. Since the key-idea goes already back to H. J. Hoppe (see \cite[Lemma 2.6]{hoppe}), this result is attributed to him. Lemma \ref{semistabilitytensorposchar} shows that Hoppe's result is also true in positive characteristic. We mention that the proof which we present below is essentially the same as the one given in \cite{brennerlookingstable}.

\begin{proposition}[Hoppe]
\label{exteriorpowercriterion}
Let $\shE$ be a vector bundle on $\PP^N$ over an algebraically closed field $K$.
Then the following holds.
\begin{enumerate}
\item The bundle $\shE$ is semistable if and only if for
every $q < \rk(\shE)$ and every $k < -q \mu(\shE)$ there does not
exist a non-trivial global section of $(\bigwedge^q \shE)(k)$.
\item If $\Gamma(\PP^N, (\bigwedge^q \shE)(k)) = 0$ for every  $q < \rk(\shE)$ and every $k \leq -q \mu(\shE)$, then $\shE$ is stable.
\end{enumerate}
\end{proposition}

\begin{proof}
We prove $(1)$. If $\shE$ is semistable, then all exterior powers
$\bigwedge^q \shE$ are also semistable by Lemma
\ref{semistabilitytensorposchar}. We have $\mu(\bigwedge^q \shE)=q \mu(\shE)$, which can be easily verified using the splitting principle (see \cite[Section I.1.2]{okonekschneiderspindler}). We have
$\Gamma(\PP^N, (\bigwedge^q \shE)(k)) = 0$ because
\begin{eqnarray*}
\mu((\bigwedge^q \shE)(k)) &=& \mu((\bigwedge^q \shE) \otimes
\mathcal{O}_{\PP^N}(k))\\&=& \mu(\bigwedge^q \shE) + \mu(\mathcal{O}_{\PP^N}(k))\\
&=& q \mu(\shE) + k\\ &<& 0
\end{eqnarray*}
and $\shE$ is semistable.

For the proof of the other direction, we do not need Lemma \ref{semistabilitytensorposchar}. Assume that for
every $q < \rk(\shE)$ and every $k < -q \mu(\shE)$ there does not
exist a non-trivial global section of $(\bigwedge^q \shE)(k)$. Let $\shF \subset \shE$ be a coherent subsheaf of rank $q < \rk(\shE)$. Then we also have an inclusion $\bigwedge^q \shF \subset \bigwedge^q \shE$. The bidual $(\bigwedge^q
\shF)^{**}$ is isomorphic to $\mathcal{O}_{\PP^N}(m)$ with $m = \deg(\shF)$. Hence, we have
$\bigwedge^q \shF \cong \mathcal{O}_{\PP^N}(m)$ outside a closed subset of
codimension $\geq 2$ (see \cite[Lemma 1.1.10]{okonekschneiderspindler}). Because $\shE$ is locally free,
there is a non-trivial sheaf morphism $\mathcal{O}_{\PP^N}(m) \rightarrow
\bigwedge^q \shE$, i.e., $\Gamma(\PP^N,(\bigwedge^q \shE)(-m)) \neq
0$. By assumption we have $-m \geq -q \mu(\shE)$ and therefore
$\mu(\shF)= m/q \leq \mu(\shE)$. Hence the vector bundle $\shE$ is semistable.

Part $(2)$ follows in the same way if we replace $<$ by $\leq$ appropriately.
\end{proof}

\begin{remark}
We recall that the concepts of semistability and stability coincide if $\deg(\shE)$ and $\rk(\shE)$ are coprime. As mentioned in \cite{bohnhorstspindler} the converse of the stability statement in Lemma \ref{exteriorpowercriterion} does not hold in general. The easiest examples are the so-called \emph{nullcorrelation bundles} on projective spaces $\PP^N$ for $N$ odd. These bundles are given by a short exact sequence
$$0 \lto \shN \lto \shT_{\PP^N}(-1) \lto \mathcal{O}_{\PP^N}(1) \lto 0,$$
where $\shT_{\PP^N}$ denotes the tangent bundle. In particular, we have $\rk(\shN)=N-1$ and $\deg(\shN)=0$. Moreover, the nullcorrelation bundles are stable and have the property that if $N \geq 5$, then $\Gamma(\PP^N, \bigwedge^2 \shN) \neq 0$ (see [ibid., Remark below Example 1.2]). So these bundles do not fulfill the exterior power condition from Lemma \ref{exteriorpowercriterion}.

Nevertheless, it is easy to see that a rank-$2$ vector bundle $\shE$ on $\PP^N$ ($N$ arbitrary) is stable if and only if $\Gamma(\PP^N,\shE(k))=0$ for $k \leq \mu(\shE)=\frac{\deg(\shE)}{2}$. Since $\Gamma(\PP^3, \shN)=0$ (see \cite[Proof of Theorem II.1.3.1(i)]{okonekschneiderspindler}), it is clear that a nullcorrelation bundle $\shN$ on $\PP^3$ is stable and fulfills the exterior power condition.
\end{remark}

\begin{remark}
\label{applicationtoothervarieties}
Let $(X,\mathcal{O}_X(1))$ be a polarized smooth projective variety of dimension $d \geq 1$ defined over an algebraically closed field of characteristic $0$.
For a semistable vector bundle $\shE$ on $X$ the numerical condition on the exterior powers $\bigwedge^q \shE$ in Lemma \ref{exteriorpowercriterion} is still fulfilled, if we replace the degree bound for the global sections by $k < -q \mu(\shE)/\deg(\mathcal{O}_X(1))$ for every $1 \leq q < \rk(\shE)$. Hence, the numerical condition is, up to the factor $1/\deg(\mathcal{O}_X(1))$, always necessary for semistability.

If additionally $\Pic(X) = \ZZ$, then the numerical criterion is again equivalent to the semistability of $\shE$. Important examples of varieties with this property are general surfaces of degree $\geq 4$ in $\PP^3_\CC$ (Noether's Theorem) and (in arbitrary characteristic) complete intersections of dimension $\geq 3$ in $\PP^N$ (see \cite[Corollary IV.3.2 and IV.4(i)]{hartshorneamplesubvarieties}).

In positive characteristic, (under the assumption $\Pic(X) = \ZZ$) the numerical condition on the exterior powers still implies semistability, but the equivalence in Lemma \ref{exteriorpowercriterion} only holds if every semistable vector bundle on $X$ is strongly semistable (see Lemma \ref{semistabilitytensorposchar}). Thus it is clear that Lemma \ref{exteriorpowercriterion} does not provide an algorithmic tool to detect semistability of vector bundles on curves. For algorithmic methods to determine semistability and strong semistability of vector bundles over an algebraic curve in positive characteristic see \cite[Chapter 3]{kaiddissertation}.
\end{remark}

\begin{example}
\label{calabiyauexample}
Let $F \in \ZZ[X_0 \komdots X_N]$, $N \geq 4$, be a homogeneous polynomial of degree $d$ such that the hypersurface $X:=\Proj(\QQ[X_0 \komdots X_N]/(F))$ is smooth. By Remark \ref{applicationtoothervarieties} we have $\Pic(X) \cong \ZZ$ and thus Hoppe's criterion \ref{exteriorpowercriterion} is applicable to determine semistability of vector bundles on $X$. Now we assume that $d= N+1$. Then the canonical bundle $\omega_X \cong \mathcal{O}_X$ is trivial which implies the semistability of the cotangent bundle $\Omega_X$ (see \cite[Theorem 3.1]{peternellsubsheaves}). In particular, $X$ is a Calabi-Yau variety. We consider $X$ as the generic fiber $\shX_0$ of the generically smooth projective morphism
$$\shX:= \Proj(\ZZ[X_0 \komdots X_N]/(F)) \lto \Spec \ZZ$$
of relative dimension $N-2$. Up to finitely many exceptions, the special fiber $\shX_p$ over a prime number $p$ is a smooth projective variety over the finite field $\FF_p$ with $\Pic(\shX_p)= \ZZ$. By the openness of semistability, the cotangent bundle $\Omega_{\shX_p}$ of the special fiber $\shX_p$ is semistable too for almost all prime numbers $p$. Since $\deg(\Omega_{\shX_p})=0$, every semistable vector bundle is strongly semistable on $\shX_p$ by \cite[Theorem 2.1]{mehtaramanathanhomogeneous}. Thus, for $p \gg 0$ we can also use Lemma \ref{exteriorpowercriterion} to detect semistability of vector bundles on $\shX_p$ (in positive characteristic).
\end{example}

%
%
%

\section{Syzygy bundles and kernel bundles} \label{syzygyandkernelbundles}

In the sequel of this article, we restrict ourselves to vector bundles on $\PP^N$, $N \geq 2$, which are kernels of surjective morphisms between splitting bundles, i.e., bundles sitting inside a short exact sequence of the form
$$0 \lto \shE \lto \bigoplus_{i=1}^n \mathcal{O}_{\PP^N}(a_i) \stackrel{\varphi}\lto \bigoplus_{j=1}^m\mathcal{O}_{\PP^N}(b_j) \lto 0,$$
where $n \geq m$. The morphism $\varphi$ is given by an $m \times n$ matrix $\shM = (a_{ji})$, where the entries $a_{ji} \in R:=K[X_0 \komdots X_N]$ are homogeneous polynomials of degrees $b_j - a_i$. In this paper, we call such a vector bundle a \emph{kernel bundle}. Special instances of kernel bundles are \emph{syzygy bundles} which correspond to the case $m=1$ and $b_1=0$, i.e., a syzygy bundle $\Syzfn$ is given by a short exact sequence
$$0 \lto \Syzfn \lto \bigoplus_{i=1}^n \mathcal{O}_{\PP^N}(-d_i) \stackrel{f_1 \komdots f_n}\lto \mathcal{O}_{\PP^N} \lto 0,$$
where $f_1 \komdots f_n \in R=K[X_0 \komdots X_N]$ are homogeneous polynomials of degrees $d_i$, $i=1 \komdots n$. If one of the polynomials is constant, the syzygy bundle $\Syzfn$ is obviously split. To exclude this case, one often demands that the
ideal $(f_1 \komdots f_n)$ is $R_+$-primary, i.e., $\sqrt{(f_1 \komdots f_n)}=R_+ = (X_0 \komdots X_N)$. The most prominent example of a syzygy bundle is the \emph{cotangent bundle} $\Omega_{\PP^N} \cong \Syz(X_0 \komdots X_N)$ of $\PP^N$.

We can compute the topological invariants of a kernel bundle $\shE$ from the presenting short exact sequence above. We have $$\rk(\shE)= n-m \mbox{ and } \deg(\shE)= c_1(\shE)=\sum_{i=1}^n a_i - \sum_{j=1}^m b_j,$$
and thus $$\mu(\shE)= \frac{1}{n-m} \left(\sum_{i=1}^n a_i - \sum_{j=1}^m b_j \right).$$
Since the Chern polynomial is multiplicative on short exact sequences, it is also easy to compute higher Chern classes of kernel bundles (see also Section \ref{sectionrestrictiontheorems}).

If $\shE$ does not split as a direct sum of line bundles, then the dual bundle $\shE^*$ of a kernel bundle has homological dimension one, and therefore we obtain the inequality $\rk(\shE) \geq N$ (see \cite[Corollary 1.7]{bohnhorstspindler}).

\begin{example}
Every vector bundle $\shE$ on the projective plane $\PP^2$ which does not split as a direct sum of line bundles has homological dimension $1$. This is easy to see, since there exists a surjective sheaf morphism $\shF \ra \shE$ for some splitting bundle $\shF = \bigoplus_{i=1}^n \mathcal{O}_{\PP^2}(a_i)$, which is also surjective on global sections (for this standard argument see for instance \cite[Lemma 1.5]{bohnhorstspindler}). Thus, there exists a short exact sequence
$$0 \lto \shK \lto \shF \lto \shE \lto 0,$$
where the first cohomology $H^1(\PP^2,\shK)$ of the kernel $\shK$ vanishes. But by the theorem of Horrocks (see \cite[Theorem I.2.3.1]{okonekschneiderspindler}) this means that $\shK$ splits as a direct sum of line bundles. Since $\shE$ is semistable if and only if $\shE^*$
is semistable, we can dualize the short exact sequence and apply Algorithm \ref{algorithmpn} to the kernel bundle $\shE^*$. Hence, our semistability algorithm is applicable to every (non-split) vector bundle on $\PP^2$ and to vector bundles of homological dimension $1$ on $\PP^N$ in general.
\end{example}

In the sequel, we show that the twists $a_1 \komdots a_n$ and $b_1 \komdots b_m$, which occur in the presenting sequence of a kernel bundle, have to fulfill a certain numerical condition which is necessary for semistability (stability). We remark that this condition
is also necessary for the semistability (stability) of kernel bundles (and in particular of syzygy bundles) on arbitrary smooth projective varieties.

If $\shE$ is a vector bundle on $\PP^N$ of rank $r$ and $\shF \subset \shE$ a subsheaf of rank $r-1$, then the quotient $\shE/\shF$ is outside codimension $2$ isomorphic to $\mathcal{O}_{\PP^N}(\ell)$ for some $\ell \in \ZZ$. This is equivalent to a section $\mathcal{O}_{\PP^N} \ra \shE^*(-\ell)$. For kernel bundles we are able to control such sections by an easy numerical condition. In particular, we can replace the condition on the global sections of the $(r-1)$th exterior power in Hoppe's criterion \ref{exteriorpowercriterion} by this condition. Before we state the result, we recall that a resolution
$$\mathfrak{F}_\bullet: 0 \lto \shF_{d} \lto \ldots \lto \shF_1 \lto \shF_0 \lto \shE \lto 0$$
of a vector bundle $\shE$ on $\PP^N$ with splitting bundles $\shF_i$, $0 \leq i \leq d$, is \emph{minimal} if the global evaluation

\begin{eqnarray*}
0 &\lto& \Gamma(\PP^N,\shF_{d}(m)) \lto \ldots \lto \Gamma(\PP^N, \shF_1(m))\\ &\lto& \Gamma(\PP^N, \shF_0(m)) \lto \Gamma(\PP^N,\shE(m)) \lto 0
\end{eqnarray*}
is exact too for every $m \in \ZZ$ and no line bundle can be omitted for two consecutive splitting bundles $\shF_i$ and $\shF_{i-1}$, which is equivalent to say that there are no constant entries ($\neq 0$) in the matrices representing the differentials in the resolution (see \cite[Section 7.2]{ottavianinotes}).

\begin{lemma}
\label{linebundlequotients}
Let $\shE$ be a vector bundle on $\PP^N$, $N \geq 2$, sitting in a short exact sequence
$$0 \lto \shE \lto \bigoplus_{i=1}^n \mathcal{O}_{\PP^N}(a_i) \lto \bigoplus_{j=1}^m \mathcal{O}_{\PP^N}(b_j) \lto 0,$$
where $a_1 \geq a_2 \geq \ldots \geq a_n$. If $a_n \geq (>) \mu(\shE) =  \frac{1}{n-m}(\sum_{i=1}^{n} a_i-\sum_{j=1}^m b_j)$, then there are no mappings from $\shE$ to line bundles which contradict the semistability $($stability$)$ of $\shE$. Moreover, if the dualized sequence is a minimal resolution for $\shE^*$, then this numerical condition is necessary for semistability $($stability$)$.
\end{lemma}

\begin{proof}
We twist the short exact sequence presenting $\shE$ with $\mathcal{O}_{\PP^N}(\ell)$ and look at the dual sequence
$$0 \lto \bigoplus_{j=1}^m \mathcal{O}_{\PP^N}(-\ell-b_j) \lto \bigoplus_{i=1}^n \mathcal{O}_{\PP^N}(-\ell-a_i) \lto (\shE(\ell))^* \cong \shE^*(-\ell) \lto 0.$$
Since $H^1(\PP^N,\bigoplus_{j=1}^m \mathcal{O}_{\PP^N}(-\ell-b_j))=0$ for all $\ell \in \ZZ$, every global section of
$(\shE(\ell))^* $ comes from $\Gamma(\PP^N, \bigoplus_{i=1}^n \mathcal{O}_{\PP^N}(-\ell-a_i))$. Consequently, for $\ell > -a_n$ there exists no non-trivial morphism $\shE(\ell) \ra \mathcal{O}_{\PP^N}$. By assumption, we have $-a_n \leq \mu(\shE^*)$.
So we have $\Gamma(\PP^N,\shE^*(-\ell))=0$ for $\ell > \mu(\shE^*)$. Hence there are no mappings to line bundles which contradict the semistability. Analogously, one obtains the corresponding statement for stability.

Next, we prove the supplement. Assume that $a_n  < (\leq)\, \mu(\shE)$. It follows from the assumption on the resolution of $\shE^*$ that the mapping $\shE \ra \mathcal{O}_{\PP^N}(a_n)$ is non-zero. But such a morphism does not exist for $\shE$ semistable or stable.
\end{proof}

\begin{remark}
\label{rankthreecase}
If $\shE$ is a rank-$3$ bundle (not necessarily a kernel bundle), then $\shE$ is semistable if and only if $\Gamma(\PP^N,\shE(\ell))=0$ for $\ell < -\mu(\shE)$ and $\Gamma(\PP^N,\shE^*(k))=0$ for $k<-\mu(\shE^*)$ (see \cite[Remark II.1.2.6]{okonekschneiderspindler}). So for kernel bundles of rank $3$ on $\PP^2$ (for $\PP^3$ see Theorem \ref{bohnhorstspindlerhdoneintro}) we only have to check global sections and the numerical condition of Proposition \ref{linebundlequotients}.
\end{remark}

\begin{remark}
It is easy to check that the subsheaf $$\ker(\bigoplus_{i \neq n} \mathcal{O}_{\PP^N}(a_i) \ra \bigoplus_{j=1}^m \mathcal{O}_{\PP^N}(b_i))$$ destabilizes the kernel bundle $\shE$ if  $a_n < \mu(\shE)$ and the resolution of $\shE^*$ is minimal.
\end{remark}

\begin{remark}
\label{minimalitybospi}
In \cite[Proposition 2.3]{bohnhorstspindler}, G. Bohnhorst and H. Spindler show that for a kernel bundle $\shE$ of rank $N$ on $\PP^N$,
the corresponding resolution of $\shE^*$ is minimal if and only if $a_1 \geq \ldots \geq a_{N+k}$, $b_1 \geq \ldots \geq b_k$ and $b_j > a_{j}$ for $j=1 \komdots k$ holds. Hence, in characteristic $0$, Theorem \ref{bohnhorstspindlerhdoneintro} shows that for kernel bundles of this type the numerical condition of Lemma \ref{linebundlequotients} is even sufficient for semistability (stability). For syzygy bundles, we give an easy characteristic free proof of Theorem \ref{bohnhorstspindlerhdoneintro}.
\end{remark}



\begin{proposition}
\label{parametersyzygybundle}
Let $K$ be a field and let $f_1 \komdots f_{N+1} \in R=K[X_0 \komdots X_N]$ be homogeneous parameters of degrees
$1 \leq d_1 \leq \ldots \leq d_{N+1}$. If $d_1 + \ldots + d_N \geq (N-1) d_{N+1}$, then the syzygy bundle
$\Syz(f_1 \komdots f_{N+1})$ is semistable on $\PP^N$. If the inequality is strict, then $\Syz(f_1 \komdots f_{N+1})$
is a stable bundle.
\end{proposition}

\begin{proof}
We use Lemma \ref{exteriorpowercriterion} to check the semistability of $\Syz(f_1 \komdots f_{N+1})$.
For a subset $I=\{i_1 \komdots i_k\} \subseteq \{1 \komdots N+1 \}$ we use the notation $d_I=\sum_{j=1}^k d_{i_j}$.
We consider the Koszul complex
$$0 \lto \shF_{N+1} \lto \shF_N \lto \ldots \lto \shF_2 \lto \shF_1 \lto \mathcal{O}_{\PP^N} \lto 0$$
on $\PP^N$ associated to the parameters $f_1 \komdots f_{N+1}$, where $$\shF_k:=\bigoplus_{|I|=k} \mathcal{O}_{\PP^N}(-d_I)$$ for $k=1 \komdots N+1$.
We have to show for every $i < N$ and every $$m < -i \mu(\Syz(f_1 \komdots f_{N+1}))= \frac{i \sum_{j=1}^{N+1}d_j}{N}$$ that $\Gamma(\PP^N, (\bigwedge^i \Syz(f_1 \komdots f_{N+1}))(m))=0$. For every $1 \leq i < N$ we have a surjection $$\shF_{i+1} \lto \bigwedge^i \Syz(f_1 \komdots f_n) \lto 0.$$ Since the Koszul complex of a regular sequence is also globally exact, every global section of $(\bigwedge^i \Syz(f_1 \komdots f_n))(m)$ comes from the bundle $\shF_{i+1}(m)$. We have $\Gamma(\PP^N,\shF_{i+1}(m))=0$ for $m < d_1+ d_2 + \ldots + d_{i+1}$. The assumption $\sum_{i=1}^N d_i \geq (N-1) d_{N+1}$ implies
$$(N-i)(d_1 + \ldots + d_{i+1}) \geq i(d_{i+2}+ \ldots + d_{N+1}),$$
for $1 \leq i \leq N-1$ (for this easy computation see \cite[Corollary 2.4]{brennerlookingstable}). But this is equivalent to $N(d_1 + \ldots + d_{i+1}) \geq i(d_1 + \ldots + d_{N+1})$, $1 \leq i \leq N-1$, and we obtain the assertion. The supplement follows analogously.
\end{proof}

\section{The semistability algorithm and its implementation} \label{algoimplementation}

The aim of this section is to use Hoppe's semistability criterion \ref{exteriorpowercriterion} to obtain a semistability algorithm for kernel bundles, which can be implemented in a computer algebra system. All implementations described in this section can be found on 

\smallskip 
http://www2.math.uni-paderborn.de/people/ralf-kasprowitz/cocoa.html. 
\medskip

The major advantage of kernel bundles compared to arbitrary vector bundles is that we can compute the global sections $\Gamma(\PP^N,\bigwedge^q \shE)$ of their exterior powers in a way which is suitable for a computer algebra system. We do not claim that such a computational approach is impossible for other vector bundles, but at least it requires more technical effort. For a kernel bundle $\shE$ we give (probably well-known) presentations of the tensor operations (tensor powers, exterior powers and symmetric powers) as kernels of mappings between splitting bundles (but these mappings are in general not surjective). This enables us to compute the global sections of these vector bundles described by applying the left exact functor $\Gamma(\PP^N, -)$. For our semistability algorithm we require such a presentation only for the exterior powers, but we will need the other tensor operations in Section \ref{tannaka}. 

\begin{proposition}
\label{exteriorpowerpresentations}
Let $\shE$ be a vector bundle on $\PP^N$, which sits in a short exact sequence
$$0 \lto \shE \lto \bigoplus_{i=1}^n \mathcal{O}_{\PP^N}(a_i) \stackrel{\shM =(a_{ji})} \lto \bigoplus_{j=1}^m\mathcal{O}_{\PP^N}(b_j) \lto 0.$$
Set $I:=\{1 \komdots n\}$, $J:=\{1 \komdots m\}$ and $Q:=\{1 \komdots q\}$ for a fixed $q \in \mathbb{N}$. \medskip

The $q$th tensor product of $\shE$  sits in the exact sequence \begin{equation} \label{formel1} 0 \lto \shE^{\otimes q} \to \bigoplus_{\alpha \in I^q}\mathcal{O}_{\PP^N}(\sum_{p \in Q} a_{\alpha_p}) \stackrel{\varphi_q} \lto \!\!\!\!\!\!\!\bigoplus_{(\beta,j,p)  \atop \beta \in I^{q-1}, j \in J, p \in Q} \!\!\!\!\!\!\! \mathcal{O}_{\PP^N}(\sum_{p \in Q-\{q\}} a_{\beta_p} + b_j),\end{equation}
where the map $\varphi_q$ is given by $$e_{\alpha} \longmapsto \sum_{(j,p)\atop  j\in J,p \in Q}a_{j,\alpha_p}e_{(\alpha^{(p)},j,p)}.$$ Here $\alpha^{(p)}$ means the $(q-1)$-tuple $\alpha$ without the $p$th element.\medskip

The $q$th exterior power of $\shE$, $1 \leq q < n-m$, sits in the exact sequence
\begin{equation} \label{formel2} 0 \lto \bigwedge^q \shE \lto \bigoplus_{A \subseteq I, |A|=q} \mathcal{O}_{\PP^N}(\sum_{i \in A} a_i) \stackrel{\varphi_q} \lto \!\!\!\!\!\!\! \bigoplus_{(B,j) \atop B \subseteq I, |B|=q-1, j \in J}   \!\!\!\!\!\!\!\!\! \mathcal{O}_{\PP^N}((\sum_{i \in B} a_i) + b_j),\end{equation}
where the map $\varphi_q$ is given by $$e_A \longmapsto \sum_{(i,j) \atop \,i \in A,\,j \in J}\sign(i,A)\, a_{ji} \,e_{(A-\{i\},j)}$$
$($the subset $A \subset I$ is supposed to have the induced order and
$$\sign(i,A) = \begin{cases} -1, \mbox{ if } i \mbox{ is an even element in } A,\\ 1, \mbox{ if } i \mbox{ is an odd element in } A). \end{cases}$$ 

Let $\Char(K) \nmid q$. The $q$th symmetric power of $\shE$ sits in the exact sequence \begin{equation} \label{formel3} 0 \lto S^q(E) \lto \!\!\!\bigoplus_{i_1 \leq \dots \leq i_q \atop  i_k \in I}\!\!\!\mathcal{O}_{\PP^N}(\sum_{k \in Q}a_{i_k}) \stackrel{\varphi_q} \lto \!\!\!\!\!\! \bigoplus_{i'_1\leq \dots \leq i'_{q-1},j\atop  i'_k \in I, j \in J}\!\!\!\!\!\!\mathcal{O}_{\PP^N}(\sum_{k \in Q-\{q\}}a_{i'_k}+b_j),\end{equation} where the map $\varphi_q$ is given by 
$$e_{i_1 \leq\dots \leq i_q} \longmapsto \sum_{i \in \{i_1,\dots, i_q\} \atop j \in J} a_{ji}e_{i_1 \leq\dots \leq \hat{i}\leq \dots \leq i_q,j}.$$Here $\hat{i}$ means that this element is omitted.


\end{proposition}  

\begin{proof}
We start with formula (\ref{formel1}). It is obviously true for $q = 1$, and assume that we proved it for some $q \in \mathbb{N}$. We consider the locally free sheaves $$C^0_q := \bigoplus_{\alpha \in I^q}\mathcal{O}_{\PP^N}(\sum_{p \in Q} a_{\alpha_p}), \;C^1_q := \!\!\!\!\! \bigoplus_{(\beta,j,p) \atop \beta \in I^{q-1}, j \in J, p \in Q}\!\!\!\!\! \mathcal{O}_{\PP^N}(\sum_{p \in Q-\{q\}} a_{\beta_p} + b_j)$$ and obtain a complex $C_q$: $\dots \to 0 \to C^0_q \stackrel{\varphi_q} \to C^1_q \to 0 \to \dots$. The complex $C_q \otimes C_1$ is by definition the complex $$\dots \to 0 \to C^0_q \otimes C^0_1 \stackrel{\varphi_q \otimes \mbox{id} \oplus \mbox{id} \otimes \varphi_1} \to C^1_q \otimes C^0_1 \oplus C^0_q \otimes C^1_1 \to C^1_q \otimes C^1_1 \to 0 \to \dots$$ see for example \cite[page 63ff]{CartanEilenberg}. Since the kernels and images of all morphisms occuring in the complex $C_1$ are locally free, we may apply the formula of K\"unneth (\cite[Theorem VI.3.1]{CartanEilenberg}) and get a short exact sequence 
\begin{eqnarray*}
0 &\lto& \bigoplus_{r+s=k}H_r(C_q) \otimes H_s(C_1) \lto H_k(C_q \otimes C_1)\\ &\lto& \bigoplus_{r+s = k-1} \shT\!\mbox{or}_1(H_r(C_q), H_s(C_1)) \lto 0,
\end{eqnarray*}
where the Tor sheaves on the right are trivial due to the fact that the sheaves $H_s(C_1)$ are locally free for all $s$. Observe that the kernel of the morphism $\varphi_q$ is $\shE^{\otimes q}$, hence we get with the K\"unneth formula $$\mbox{ker}(\varphi_q \otimes \mbox{id} \oplus \mbox{id} \otimes \varphi_1) = H_0(C_q \otimes C_1) \cong H_0(C_q) \otimes H_0(C_1) = \shE^{\otimes q} \otimes \shE. $$ Furthermore, it is not difficult to obtain the identity $\varphi_{q+1} = \varphi_q \otimes \mbox{id} \oplus \mbox{id} \otimes \varphi_1$ under the obvious identifications $C_q^0 \otimes C_1^0 \cong C_{q+1}^0$ and $C_q^1 \otimes C_1^0 \oplus C^0_q \otimes C^1_1 \cong C^1_{q+1}$.   

For formula (\ref{formel2}), let $0 \ra \shE \ra \shF \stackrel{\varphi} \ra \shG \ra 0$ be a short exact sequence of locally free sheaves. Consider the Koszul complex of the $\mathcal{O}$-algebra $S\shG$, locally defined by the sequence of elements in $S\shG$ consisting of a basis of $\shF$, i.e., the complex $$0 \lto \bigwedge^s\shF \otimes S\shG \stackrel{d_s} \lto \bigwedge^{s-1}\shF \otimes S\shG \stackrel{d_{s-1}}\lto \dots \lto \shF \otimes S\shG \stackrel{d_1}\lto S\shG \lto 0,$$ with $s = \rk(\shF)$ and $$d_j(f_{i_1}\wedge\dots\wedge f_{i_j} \otimes g) = \sum_{k=1}^j(-1)^{k-1} (f_{i_1} \wedge \ldots \wedge \hat{f}_{i_k} \wedge \ldots \wedge f_{i_j}) \otimes f_{i_k}\cdot g$$ ($\hat{f}_{i_k}$ means that we omit the $k$-th factor). One easily checks that for every $1 \leq q < \rk(\shE)$ this yields an exact complex $$0 \lto \bigwedge^q  \shE \lto \bigwedge^q \shF
\stackrel{\varphi_q} \lto  \bigwedge^{q-1} \shF  \otimes \shG \lto \dots \lto \shF \otimes S^{q-1}\shG \lto S^q\shG \lto 0,$$ with $\varphi_q := d_q$. Applying this argument for the short exact sequence defining the kernel bundle $\shE$ together with the isomorphism $ \bigwedge^r\bigoplus_{i \in I}\mathcal{O}_{\PP^N}(a_i) \cong \bigoplus_{A \subseteq I, |A|=q} \mathcal{O}_{\PP^N}(\sum_{i \in A} a_i)$ yields the claim. 

The case (\ref{formel3}) works analogously to the second one, but now we take the Koszul complex of $S\shF$ locally with respect to the regular sequence consisting of a basis of $\shE$. This gives the exact complex $$0 \ra \bigwedge^r\shE \otimes S^{q-r}\shF \lto \bigwedge^{r-1}\shE \otimes S^{q-2}\shF \lto \dots \lto \shE \otimes S\shF \lto S\shF \lto S\shG \ra 0.$$ Applying this to the dual exact sequence that defines the kernel bundle $\shE$ and dualizing again yields the exact complex $$0 \lto S^q(\shE^*)^* \lto S^q(\bigoplus_{i \in I}\mathcal{O}_{\PP^N}(a_i)) \stackrel{d_1^*} \lto \bigoplus_{j \in J}\mathcal{O}_{\PP^N}(b_j) \otimes S^{q-1}(\bigoplus_{i \in I}\mathcal{O}_{\PP^N}(a_i))$$ Furthermore, there are isomorphisms $$S^q(\bigoplus_{i \in I}\mathcal{O}_{\PP^N}(a_i)) \cong \bigoplus_{i_1\leq\dots \leq i_q \atop i_k \in I}\mathcal{O}_{\PP^N}(\sum_{k \in Q}a_{i_k})$$ and  $S^q(\shE^*) \cong S^q(\shE)^*$ if the characteristic of the field $K$ does not divide $q$, see \cite[Satz 86.12]{schejastorch2}.
\end{proof}

Now, we describe the implementation of the semistability algorithm for kernel bundles.

\begin{remark}
\label{globalsectionconecomputation}
Let $\shE$ be a vector bundle on $\PP^N$. The twisted dual bundle $\shE^*(n)$ is generated by global sections for $n \gg 0$ (see \cite[Theorem 5.17]{hartshornealgebraic}), i.e.,
we have a surjection $\mathcal{O}_{\PP^N}(-n)^s \ra \shE^* \ra 0$ for some $s$. Let $\shK$ be the kernel of this morphism. Taking duals yields a short exact sequence
$$0 \lto \shE \lto \mathcal{O}_{\PP^N}(n)^s \lto \shK^* \lto 0$$
of vector bundles. As in the proof of Proposition \ref{exteriorpowerpresentations}, we derive for every $1 \leq q \leq \rk(\shE)-1$ an  exact sequence
$$0 \lto \bigwedge^q \shE \lto \mathcal{O}_{\PP^N}(qn)^{{s \choose q}} \stackrel{\varphi_q}\lto \mathcal{O}_{\PP^N}((q-1)n)^{{s \choose {q-1}}} \otimes \shK^*.$$
Since we also have a surjection $\mathcal{O}_{\PP^N}(-m)^t \ra \shK \ra 0$ for $m \gg 0$ and some $t$, we obtain a diagram

\smallskip

\begin{xy}
\xymatrix{
&  &  &  & 0 \ar[d] \\
&0 \ar[r] & \bigwedge^q \shE \ar[r] & \mathcal{O}_{\PP^N}(qn)^{{s \choose q}} \ar[dr]^{\bar{\varphi}_q} \ar[r]^{\hspace{-0.8cm}\varphi_q} & \mathcal{O}_{\PP^N}((q-1)n)^{{s \choose {q-1}}} \otimes \shK^* \ar[d]\\
&  & & & \mathcal{O}_{\PP^N}((q-1)n)^{{s \choose {q-1}}} \otimes \mathcal{O}_{\PP^N}(m)^t, }
\end{xy}

\bigskip

\noindent where the maps $\varphi_q$ and $\bar{\varphi}_q$ have the same kernel. In this sense, the algorithmic methods to determine semistability which we develop in this chapter, are applicable to every vector bundle on $\PP^N$.
\end{remark}

\begin{remark}
\label{cokernelpresentation}
Let $M=\bigoplus_{d \in \ZZ} M_d$ be a finitely generated graded $R$-module ($R=K[X_0\komdots X_N]$). If we fix homogeneous generators $g_1 \komdots g_n$, we obtain a surjection
$$R^n \stackrel{f} \lto M,~ e_j \longmapsto g_j,~ j=1 \komdots n.$$
From this map we can derive, for every positive integer $q\geq 1$, the well-known exact sequence
$$ (\ker f) \otimes_R \bigwedge^{q-1} R^n \stackrel{\alpha} \lto \bigwedge^q R^n \lto \bigwedge^q M \lto 0,$$
where the map $\alpha$ is given by
$$x \otimes (y_1 \wedge \ldots \wedge y_{q-1}) \longmapsto x \wedge y_1 \wedge \ldots \wedge y_{q-1}$$
(see \cite[\S 83, Aufgabe 26]{schejastorch2}). Fixing homogeneous generators of $\ker f$ gives a diagram

\smallskip

\begin{xy}
\xymatrix{
& 0 &  &  & \\
& {(\ker f) \otimes_R \bigwedge^{q-1} R^n} \ar[r]^{\hspace{0.8cm}\alpha} \ar[u] & {\bigwedge^q R^n} \ar[r] & \bigwedge^q M \ar[r] & 0  \\
& R^m \otimes_R  \bigwedge^{q-1} R^n. \ar[u] \ar[ur]^{\bar{\alpha}} & & &  }
\end{xy}

\bigskip

\noindent In this way we obtain a presentation of $\bigwedge^q M$ as a cokernel of a map between free modules. Since all these mappings are graded, we have a corresponding sequence
$$\bigoplus_{i=1}^{\tilde{n}} \mathcal{O}_{\PP^N}(a_i) \stackrel{\bar{\alpha}}\lto \bigoplus_{j=1}^{\tilde m} \mathcal{O}_{\PP^N}(b_j) \lto \bigwedge^q \widetilde{M} \lto 0,$$
of the associated coherent sheaves on $\PP^N$ (with suitable twists). But the map
$\bigoplus_{j=1}^{\tilde{m}} R_{b_j} \ra \Gamma(\PP^N,\bigwedge^q \widetilde{M})$ is in general not surjective. Hence this sequence cannot be the basis of an algorithmic approach. If the depth of $M$ is at least $2$, then this map is surjective if and only if  $\bigwedge^q(\Gamma(\PP^N,\widetilde{M})) \ra \Gamma(\PP^N, \bigwedge^q \widetilde{M})$ is surjective (see \cite[Theorem A4.1 and Theorem A4.3]{eisenbud}). An illustrating example is the following.
\end{remark}

\begin{example}
\label{monomialexample}
We consider the syzygy bundle $\shS:=\Syz(X^3,Y^3,Z^3,XY^2Z^2)$ on $\PP^2= \Proj K[X,Y,Z]$ (see also \cite[Example 7.4]{brennerlookingstable}) and use Brenner's criterion \ref{monomialcase}. The slope of this bundle equals $-\frac{14}{3} \approx -4,66$. For the subsheaves of rank $1$ coming from two monomials we have (we only list the combinations having a common factor)
\begin{eqnarray*}
\mu(\Syz(X^3,XY^2Z^2))&=&1-8=-7,\\ \mu(\Syz(Y^3,XY^2Z^2))&=&2-8=-6 \mbox{ and}\\ \mu(\Syz(Z^3,XY^2Z^2))&=&2-8=-6.
\end{eqnarray*}
Hence we see that the global sections of $\shS$ do not contradict the semistability. But the monomial subfamily $X^3,Y^3,Z^3$ yields the rank-$2$ subbundle $\Syz(X^3,Y^3,Z^3) \subset \shS$ of slope $-\frac{9}{2}=-4,5$.
Thus, $\shS$ is not semistable with Harder-Narasimhan filtration
$$0 \lto \Syz(X^3,Y^3,Z^3) \lto \shS \lto \mathcal{O}_{\PP^2}(-5) \lto 0.$$
Moreover, the mapping $\bigwedge^2(\Gamma(\PP^2,\shS)) \ra \Gamma(\PP^2,\bigwedge^2 \shS)$ is not surjective.
\end{example}

Since we assume that a vector bundle and its exterior powers are given as kernels of morphisms between splitting bundles, we have to know how to compute the kernel of an $R$-linear map $R^n \ra R^m $ between finitely generated free modules over the polynomial ring $R$. The answer is given by the following well-known lemma which shows that we can compute a minimal nontrivial global section with Gr\"obner bases.

\begin{lemma}
\label{kernelcomputation}
Let $R=K[X_0 \komdots X_N]$ be the polynomial ring over a field $K$ and let $\varphi: R^m \ra R^n$ be an $R$-linear map. Denote by $e_1 \komdots e_m$ the standard basis vectors of $R^m$. With the notation $w_j = \varphi(e_j)$, $j=1 \komdots m$, we have $$\ker \varphi =\Syz_R(w_1 \komdots w_m).$$ In other words, the kernel of $\varphi$ is the $(R$-$)$syzygy module of the columns of the matrix which describes $\varphi$.
\end{lemma}

\begin{proof}
See \cite[Proposition 3.3.1(a)]{kreuzerrobbiano}.
\end{proof}

So we have accumulated the necessary technical tools to formulate an algorithm, based on Hoppe's criterion \ref{exteriorpowercriterion}, to determine semistability of a given kernel bundle on $\PP^N$. Note that every instruction can be performed with any computer algebra system which is able to handle Gr\"obner bases calculations. In our case, we implemented it in CoCoA \cite{CocoaSystem}.

\begin{algorithm}[Semistability of kernel bundles] \label{semistability}
\label{algorithmpn}
\mbox{}

\smallskip

\noindent \underline{Input:} Two lists $[a_1 \komdots a_n]$, $[b_1 \komdots b_m]$ and a homogeneous $m \times n$ matrix $\shM=(a_{ji})$ with no constant polynomial entries $a_{ji}\neq 0$ of degrees $b_j- a_i$ defining a kernel bundle
$$0 \lto \shE = \widetilde{\ker \shM} \lto \bigoplus_{i=1}^n \mathcal{O}_{\PP^N}(a_i) \stackrel{\shM}\lto \bigoplus_{j=1}^m\mathcal{O}_{\PP^N}(b_j) \lto 0$$
with $a_1 \geq a_2 \geq \ldots \geq a_n$ $($as usual $\widetilde{\ker \shM}$ denotes the sheaf associated to the graded $R$-module $\ker \shM)$.
\bigskip

\noindent \underline{Output:} The decision whether $\shE$ is semistable in terms of a boolean value TRUE or FALSE respectively.

\begin{enumerate}
\item Compute the invariants $\rk(\shE)=n-m$, $\deg(\shE)= \sum_{i=1}^n a_i - \sum_{j=1}^m b_j$ and $\mu(\shE)= \frac{1}{n-m} \left(\sum_{i=1}^n a_i - \sum_{j=1}^m b_j \right)$.
\item If the slope condition $a_n \geq \mu(\shE)$ of Proposition \ref{linebundlequotients} is fulfilled, then continue. Else return FALSE and terminate.
\item Set $q:=1$.
\item Construct the matrix $\shM_q$ which describes the map $\varphi_q$ in Proposition \ref{exteriorpowerpresentations}.
\item Compute the syzygy module $S_q$ of the columns of $\shM_q$.
\item Compute the initial degree $\alpha_q:=\min \{t: (S_q)_t \neq 0 \}$ of the graded $R$-module $S_q$ $($i.e., $\alpha_q$ is the minimal twist $\ell$ such that $\Gamma(\PP^N, (\bigwedge^q \shE)(\ell)) \neq 0)$.
\item If $\alpha_q < -q\mu(\shE)$, then return FALSE and terminate. Else set $q:=q+1$ and continue.
\item If $q< \rk(\shE)-1$, then go back to step $(4)$.
Else return TRUE and terminate.
\end{enumerate}
\end{algorithm}

\begin{remark}
We require the lists $[a_1 \komdots a_n]$ and $[b_1 \komdots b_n]$ as input data, since they define the degree of every entry $a_{ji}$ in the matrix $\shM$. This becomes clear if the matrix $\shM$ contains zeros (which is often the case if $m \geq 2$). An illustrating example is $\Syz(f_1,f_2,f_3,0)$ on $\PP^2$ for $R_+$-primary elements $f_1,f_2,f_3$. Since $$\Syz(f_1,f_2,f_3,0) \cong \Syz(f_1,f_2,f_3) \oplus \mathcal{O}_{\PP^2}(-d),$$
the semistability heavily depends on the degree $d$ which we have allocated to $0$.
\end{remark}

\begin{remark}
One easily verifies that the map given by the matrix $\shM$ is surjective if and only if the ideal generated by all $m \times m$ minors of $\shM$ is $R_+$-primary. The latter property can be checked computationally.
\end{remark}

We give a first example of an application of our semistability algorithm.

\begin{example}
Let $K$ be an arbitrary field. We consider the monomials $X^2,Y^2,XY,XZ,YZ \in R = K[X,Y,Z]$ and the corresponding sheaf of syzygies $\shS:=\Syz(X^2,Y^2,XY,XZ,YZ)$. Is $\shS$ a semistable sheaf? Since the ideal generated by these monomials is not $R_+$-primary, we can neither apply Theorem \ref{monomialcase} nor (at first sight) Algorithm \ref{algorithmpn}. We compute a resolution of $\shS$ (for instance with CoCoA), namely,
$$0 \lto \mathcal{O}_{\PP^2}(-4)^2 \stackrel{\shA}\lto  \mathcal{O}_{\PP^2}(-3)^6 \lto \shS \lto 0,$$
where
$$\shA= \left( \begin{array}{cc}
x & 0 \\
 - y & 0 \\
 - y & x \\
0 &  - y \\
 - z & 0 \\
0 & z \end{array}\right).$$
Since $\shS$ is a reflexive sheaf, it is locally free on $\PP^2$ (cf. \cite[Lemma 1.1.10]{okonekschneiderspindler}).
So if we dualize the resolution, we obtain a short exact sequence
$$0 \lto \shS^* \lto  \mathcal{O}_{\PP^2}(3)^6 \stackrel{\shA^t}\lto \mathcal{O}_{\PP^2}(4)^2 \lto 0,$$
i.e., $\shS^*$ is a kernel bundle. Hence, we apply Algorithm \ref{algorithmpn} to $\shS^*$ in order to obtain the answer to our question
(we recall that $\shS$ is semistable if and only if $\shS^*$ is semistable). A CoCoA computation gives:
\begin{enumerate}
\item $H^0(\PP^2,\shS^*(m))=0$ for $m < -2$ and $-2 \geq -\mu(\shS^*)=-\frac{5}{2}$.
\item $H^0(\PP^2,(\bigwedge^2(\shS^*))(m))=0$ for $m<-5$ and $-5 =-2 \mu(\shS^*)$. In particular, we obtain no information about stability.
\item The numerical condition of Proposition \ref{linebundlequotients} is fulfilled. So there are no mappings $\shS^* \ra \mathcal{O}_{\PP^2}(k)$ into line bundles which contradict the semistability.
\end{enumerate}
Eventually, we conclude via Lemma \ref{exteriorpowercriterion} that $\shS^*$ is semistable and so is $\shS$.
\end{example}

\section{Tannaka duality of stable syzygy bundles} \label{tannaka}

As a first application of the algorithms described in the previous sections we will compute the Tannaka dual groups of certain stable syzygy bundles of degree $0$ on the projective plane. We start with describing the setting. From now on, $K$ denotes an algebraically closed field of characteristic $0$ and $X$ a smooth, irreducible and projective variety over $K$. Furthermore, denote by $\mathfrak{B}_X$ the category of polystable vector bundles of degree $0$ on $X$. It is well-known that $\mathfrak{B}_X$ is an abelian tensor category that possesses the faithful fiber functor $\omega_x: \mathfrak{B}_X \to Vect(K)$, where $Vect(K)$ is the category of finite-dimensional $K$-vector spaces and $\omega_x$ maps a bundle $E$ to its fiber $\shE_x$ for a point $x \in X(K)$. In other words, it is a neutral Tannaka category and hence there exists an affine group scheme $G_X$ over $K$ and an equivalence of categories $$\mathfrak{B}_X \stackrel{\sim} \longrightarrow \mathbf{Rep}_{G_X}(K).$$ For the theory of Tannaka categories, see e.g. \cite{delignemilne}. We denote by $\mathfrak{B}_{\shE}$ the smallest Tannaka subcategory of $\mathfrak{B}_X$ containing the vector bundle $\shE$ and by $G_{\shE}$ its Tannaka dual group. The group scheme $G_X$ is pro-reductive and $G_{\shE}$ is a reductive linear algebraic group (not necessarily connected). It is in a natural way an algebraic subgroup of $\mbox{GL}_{\shE_x}$. Furthermore, there is a faithfully flat morphism $G_X \to G_{\shE}$.
Since global sections of vector bundles in $\mathfrak{B}_X$ correspond to $G_{\shE}$-invariant elements of the fiber, it follows from \cite[Proposition 3.1]{delignemilne} that the algebraic group $G_{\shE}$ is uniquely determined by the global sections of $T^{r,s}(\shE) := \shE^{\otimes r}\otimes (\shE^*)^{\otimes s}$ for $r,s \in \mathbb{N}$. It even suffices to know the global sections of $\shE^{\otimes r}$ for $r \in \mathbb{N}$, since the dual of a stable bundle $\shE$ occurs as a direct summand in some tensor power of $\shE$. 

\begin{remark}
\label{nottannakian}The restriction to fields of characteristic $0$ is essential, as the following example shows. It was communicated to us by H. Brenner. 
Let $K$ be a field of positive characteristic $p$, $p \geq 3$. Consider the plane curve $C=V_+(X^{3p-1}+Y^{3p-1}+Z^{3p-1}+X^p Z^{2p-1})$. This curve is smooth by the Jacobian criterion. Now we look at the syzygy bundle $\shE:=\Syz(X^2,Y^2,Z^2)(3)$ on $C$ of degree $0$. Since $\shE$ is stable on $\PP^2$  by Proposition \ref{parametersyzygybundle} and $p \geq 3$, it remains stable on $C$ by Langer's restriction theorem \ref{langerrestrictiontheorem}.
The Frobenius pull-back $F^*(\shE) \cong \Syz(X^{2p},Y^{2p},Z^{2p})(3p)$ has the non-trivial section $s:=(ZX^{p-1},ZY^{p-1},Z^p+X^p)$ because we have the equation
\begin{eqnarray*}
X^{2p} \cdot ZX^{p-1} + Y^{2p} \cdot ZY^{p-1} + Z^{2p} \cdot (Z^p+X^p) &=& \\ Z (X^{3p-1}+Y^{3p-1}+Z^{3p-1}+X^p Z^{2p-1}) &=& 0
\end{eqnarray*}
on the curve. It is easy to see that $s$ has no zeros on $C$ and that there is no further non-trivial section of $F^*(\shE)$. Hence $F^*(\shE)$ is a non-trivial extension of the structure sheaf by itself and therefore not polystable. Since $F^*(\shS) \subset S^p(\shE)$, it follows that $S^p \shE$ is not polystable either. The same holds for $\shE^{\otimes p}$ since $S^p \shE$ is a quotient of the $p$-fold tensor product. So we see that
$\mathfrak{B}_C^s$ is not a tensor-category and in particular not Tannakian.
\end{remark}

Let us now consider stable bundles $\shE$ of degree $0$ on the projective space $\PP^N$.
\begin{lemma} \label{connectedlemma}
The Tannaka dual group $G_{\shE}$ of a stable vector bundle $\shE$ of degree $0$ on $\PP^N$ is a connected semisimple group.
\end{lemma}
 
\begin{proof}
Suppose that the algebraic group $G_{\shE}$ is not connected. Then the representations of the finite quotient $G_{\shE}/G_{\shE}^0$ would correspond to a subcategory of $\mathfrak{B}_{\shE}$ containing nontrivial finite vector bundles, see \cite[Lemma 3.1]{nori}. But the latter form, together with the obvious fiber functor, a neutral Tannaka category (see \cite[Proposition 3.7]{nori}), and the $K$-valued points of the Tannaka dual group are well known to coincide with the \'etale fundamental group if the characteristic of the ground field is $0$. Hence there are no nontrivial finite vector bundles on the projective space. Furthermore, the reductive group $G_{\shE}$ does not have any nontrivial characters due to the fact that $\mbox{Pic}(\PP^N) = \ZZ$, hence it has to be semisimple.
\end{proof}

It was shown in \cite[Lemma 4.4 and Proposition 5.3]{monodromygroups} that for stable vector bundles of degree $0$ and rank $r$ as in Theorem \ref{bohnhorstspindlerhdoneintro} the almost simple components of the Tannaka dual group have to be of type $A$. We conjecture that it is always the standard $r$-dimensional representation of $\mathbf{SL}_r$ or its dual. One motivation for this paper was to construct examples of syzygy bundles on the projective space having a Tannaka dual group of type different from $A$. Note that one cannot simply try to guess an example for a syzygy bundle with group different from $\mathbf{SL}_r \subset \mathbf{GL}_r$, see the remark \ref{moduli} below. Our idea is to exclude this case by constructing self-dual bundles. With the algorithmic methods described in \ref{semistability} and Proposition \ref{exteriorpowerpresentations}, it is in principle possible to compute the Tannaka dual group and its representation for an arbitrary stable kernel bundle of degree $0$. However, the necessary computations grow very fast with the rank of the bundle, so we were only able to handle syzygy bundles up to rank $6$ on $\PP^2$. Furthermore, we only found syzygy bundles having the almost simple Tannaka dual group $\mathbf{Sp}_r \subset \mathbf{GL}_r$, where $r \in \{4,6\}$. There are no stable self-dual syzygy bundles of odd rank on $\PP^2$, see Corollary \ref{nooddrank}.\medskip

For stable rank $2$ bundles of degree $0$ on $\PP^2$, there is only one possible Tannaka dual group, namely the $2$-dimensional irreducible representation of $\mathbf{S}\mathbf{L}_2$. An example for this is the syzygy bundle $\mbox{Syz}(X^2,Y^2,Z^2)(3)$. It is stable due to Theorem \ref{bohnhorstspindlerhdoneintro}. To find higher rank syzygy bundles whose Tannaka dual group is not the group $\mathbf{SL}_r$ with an $r$-dimensional representation, we will use the following simple Lemma.

\begin{lemma} \label{selfdual}
 Let $f_1,\dots,f_n \in R:= K[X,Y,Z]$ be homogeneous polynomials  such that the ideal $I:=(f_1,\dots,f_n)$ is $R_+$-primary and minimally generated by $f_1,\dots, f_n$. Then $\shE:= \Syz(f_1,\dots,f_n)$ is self-dual $($up to a twist with a line bundle$)$ if and only if $R/I$ is Gorenstein. 
\end{lemma}
 
\begin{proof}
The minimal free resolution of $R/I$ has length $3$ and is self-dual up to twist since $R/I$ is Gorenstein: $$ 0 \longrightarrow  R(-d) \stackrel{\varphi}{\longrightarrow}  \bigoplus_{i = 1}^nR(-e_i) \longrightarrow  \bigoplus_{i = 1}^nR(-d_i) \stackrel{f_1,\dots,f_n}{\longrightarrow}   R \longrightarrow  R/I \longrightarrow 0$$ and with $E:= \mbox{ker}(f_1,\dots, f_n)$ we have $\mbox{coker}(\varphi) = E(-d)^*$. In particular, there is an isomorphism $E \cong E(-d)^*$ with $\shE = \widetilde{E}$. Conversely, one easily sees that if $\shE$ is self-dual up to twist, then the minimal free resolution of $R/I$ ends with a free module of rank $1$, that is $R/I$ is Gorenstein. 
\end{proof}

\begin{corollary} \label{nooddrank}
 There are no self-dual $($up to a twist with a line bundle$)$ non-split syzygy bundles of odd rank on $\PP^2$. 
\end{corollary}

\begin{proof}
 This is \cite{buchsbaumeisenbud}, Corollary 2.2, which says that the minimal number of generators of a Gorenstein ideal of grade $3$ is odd.
\end{proof}

\begin{corollary}
 All stable syzygy bundles of degree $0$ and odd rank $\leq 11$ on the projective plane have a semisimple Tannaka dual group whose simple components are of type $A$.
\end{corollary}

\begin{proof} A table of representations of Lie algebras (e.g. \cite{tableslie}) shows that the smallest non-self-dual irreducible representation of a semisimple algebraic group with some simple component not of type $A$ is the $($up to duality$)$ $12$-dimensional representation of the semisimple group $\mathbf{S}\mathbf{L}_3 \oplus \mathbf{S}\mathbf{p}_4$, which is the tensor product of the $3$-dimensional irreducible representation of $\mathbf{S}\mathbf{L}_3$ with the $4$-dimensional irreducible representation of $\mathbf{S}\mathbf{p}_4$.
\end{proof}

Now looking at rank $4$, what possibilities are there for the Lie algebra of the Tannaka dual group? There are three different self-dual and irreducible representations, where we use the notations of the tables of simple Lie algebras and their representations in \cite{tableslie}: $A_1$ with highest weight $3$, $C_2$ with highest weight $(1,0)$, and $A_1 \oplus A_1$ with highest weight $(1,1)$. Using the computer algebra package for Lie group computations LiE \cite{LiE}, one finds that $\dim(\Gamma(\PP^2, \shE^{\otimes 4}))= \dim((\shE_x^{\otimes 4})^{G_{\shE}}) = 4$ in the cases of type $A$ and $\dim(\Gamma(\PP^2, \shE^{\otimes 4}))=3$ in the case of $C_2$. To find a bundle with Tannaka dual group of type $C_2$, we have to construct a Gorenstein ideal $I$ with $5$ minimal generators, such that the associated syzygy bundle $\shE$ has degree $0$. Then we have to check its stability and the global sections of $\shE^{\otimes 4}$. \medskip

To find a suitable Gorenstein ideal $I$, we consider the polynomial ring $R:=K[X,Y,Z]$ as a module over itself by interpreting a polynomial in $R$ as differential operator acting on itself, e.g., $X\cdot f = \frac{\partial f}{\partial X}$. Choose a homogeneous polynomial $f$ of degree $r$ and define $I := \mbox{Ann}_R(f) \subseteq R$. This ideal is Artinian and Gorenstein, see for example \cite{gorensteinalgebras}, Lemma 2.12. 

\begin{example} \label{exampleofstablebundle}
 The Gorenstein ideal of the homogeneous form $$f:= X^2+Y^2+Z^2$$ is the ideal $I = (X^2-Y^2,X^2-Z^2,XY,XZ,YZ)$. If we pull back the associated syzygy bundle on $\PP^2$ via the finite morphism $X \mapsto X^2,Y \mapsto Y^2, Z \mapsto Z^2$, we obtain the twisted syzygy bundle $$\shE(5):= \Syz(X^4-Y^4,X^4-Z^4,X^2Y^2,X^2Z^2,Y^2Z^2)(5).$$ It has degree $0$, rank $4$ and is self-dual by Lemma \ref{selfdual}. We apply Algorithm \ref{algorithmpn} to $\shE$ and obtain with the help of CoCoA:
\begin{enumerate}
\item $H^0(\PP^2,\shE(m)) = 0$ for $m \leq 5= -\mu(\shE)$;
\item $H^0(\PP^2,(\bigwedge^2\shE)(m))=0$ for $m <10=-2\mu(\shE)$,  $H^0(\PP^2,(\bigwedge^2\shE)(m))\neq 0$ for $m=10$;
\item $H^0(\PP^2,(\bigwedge^3\shE)(m))=0$ for $m \leq 15=-3\mu(\shE)$ (this computation is, by Proposition \ref{linebundlequotients}, actually not necessary since the degrees of the polynomials are constant).
\end{enumerate}
Hence, we see that $\shE$ $($and all its twists$)$ are semistable, but we get no information about stability because of $($2$)$. Observe that Hoppe's criterion can never reveal stability of a self-dual bundle of rank $4$, since in this case there must be nontrivial global sections of $\Lambda^2(\shE(5))$, see again \cite{tableslie}. \medskip

Fortunately, the self-duality of the bundle allows us to prove its stability. By the computation above, we only have to consider subsheaves of rank $2$ which may destroy the stability of $\shE(5)$. Assume that $\shF \subset \shE(5)$ is a stable subsheaf of rank $2$ and degree $0$. Since we can pass over to the reflexive hull, and since we are working on $\PP^2$, we may assume
that $\shF$ is locally free (see \cite[Lemma 1.1.10]{okonekschneiderspindler}). In particular, we have
$$\shF \cong \shF^* \otimes \det(\shF) \cong \shF^* \otimes \mathcal{O}_{\PP^2} \cong \shF^*$$ by \cite[Proposition 1.10]{hartshornestablereflexive}. That is, the subsheaf $\shF$ is self-dual too. So the composition of the morphisms
$$\shE(5) \cong \shE(5)^* \lto \shF^* \cong \shF \hookrightarrow \shE(5)$$
yields an endomorphism of $\shE(5)$, which is not a multiple of the identity. But a computation of global sections using the implementation of Proposition \ref{exteriorpowerpresentations} yields that $$h^0(\PP^2,\End(\shE(5))) = h^0 (\PP^2, \shE(5) \otimes \shE(5)) =1,$$ that is, the bundle $\shE(5)$ is simple and a morphism as above does not exist. Hence, the bundle $\shE(5)$ is stable. Finally, another computation shows $h^0(\PP^2,(\shE(5)^{\otimes 4})) = 3$, hence the Tannaka dual group is in fact almost simple of type $C_2$ and the representation of its Lie algebra has highest weight $(1,0)$. It is well-known that this corresponds to the irreducible and faithful representation $\mathbf{Sp}_4 \subset \mathbf{GL}_4$.
\end{example}

\begin{example} \label{secondexample}
The Gorenstein ideal associated to the homogeneous form $X^3Y+Y^3Z+Z^3X$ via the correspondence described above is the ideal $$I := (X^3-Y^2Z,Y^3-XZ^2,X^2Y-Z^3,XY^2,YZ^2,X^2Z,XYZ).$$ We consider the same pull-back as in the previous example to be able to twist the associated syzygy bundle to degree $0$. The same computations as above show that the corresponding syzygy bundle $\shE(7)$, defined as $$\Syz(X^6-Y^4Z^2,Y^6-X^2Z^4,X^4Y^2-Z^6,X^2Y^4,Y^2Z^4,X^4Z^2,X^2Y^2Z^2)(7)$$ is semistable of degree $0$, where a subsheaf $\shF$ that destroys stability has to be of rank $r=4$ or $r=2$. But since $\shE(7)$ is self-dual, we can always assume $\shF$ to be of rank $2$ and obtain stability for the same reason as above, since the bundle again turns out to be simple. Again using LiE \cite{LiE}, we find that it is possible to determine the Lie algebra of the Tannaka dual group by computing $\dim(\shE^{\otimes 4}) = 3$, which shows that it has to be simple of type $C_3$, with highest weight $(1,0,0)$. This corresponds to the faithful irreducible representation $\mathbf{Sp}_6 \subset \mathbf{GL}_6$.
\end{example}    

\begin{remark} \label{moduli}
There is a good reason why one should expect to find the group $\mathbf{Sp}_{r}$ in these cases. It is well-known that the moduli space of stable bundles of fixed rank and Chern classes exists as a quasi-projective variety, see for example \cite[Theorem 4.3.4]{huybrechtslehn}. Let us denote by $M$ the moduli space containing the syzygy bundle of Example \ref{exampleofstablebundle} respectively of Example \ref{secondexample}. Let $\mathcal{U}$ be the quasi-universal bundle on $\PP^2 \times M$, see \cite[Chapter 4.6]{huybrechtslehn}. This means that the restriction of $\mathcal{U}$ to $\PP^2 \times \{p\}$ is a finite product of copies of the stable bundle on $\PP^2$ corresponding to the point $p$. Applying the semicontinuity theorem $($ e.g. \cite[Theorem 12.8]{hartshornealgebraic} $)$ for $\mbox{dim}(H^0(\PP^2 \times \{p\}, \mathcal{U}^{\otimes r}))$ one finds that the locus of the vector bundles having Tannaka dual group $\mathbf{SL}_r$ is open in $M$ since the dimension of $\Gamma(\PP^2 \times \{p\}, \mathcal{U}^{\otimes r})$ is minimal in this case and strictly higher in all other cases. Hence for a generic choice of a stable syzygy bundle on $\PP^2$ one expects to find the group $\mathbf{SL}_r$ as Tannaka dual group. The locus of self-dual bundles is closed in $M$ (apply the semicontinuity theorem for $\mathcal{U} \otimes \mathcal{U}$), containing the bundles with Tannaka dual group $\mathbf{Sp}_r$ as an open locus (since the dimension of $\Gamma(\PP^2 \times \{p\}, \mathcal{U}^{\otimes 4})$ is minimal for a self-dual stable bundle with this Tannaka dual group, see the discussion of the examples above). It follows that for  a generic choice of a self-dual syzygy bundle one expects the group $\mathbf{Sp}_r$ as Tannaka dual group. It would be interesting to find a method for constructing stable syzygy bundles having a Tannaka dual group different from $\mathbf{SL}_r$ or $\mathbf{Sp}_r$. Furthermore, one could try to determine the geometry of the Tannaka-strata of the moduli spaces discussed above.
\end{remark}

\section{Restriction theorems -- a brief overview} 
\label{sectionrestrictiontheorems}

Algorithm \ref{algorithmpn} enables us to determine semistability (and in some cases stability) of kernel bundles on projective spaces. If our algorithm gives a positive answer, then what can be said about the semistability (respectively stability) of these bundles when we restrict them to a hypersurface $X \subset \PP^N$?  The answer is given by so-called \emph{restriction theorems} which ensure the semistability (stability) of the restriction of a bundle to hypersurfaces of sufficiently large degree. These theorems hold in more general situations, we only give formulations for vector bundles on projective spaces. By applying restriction theorems, we can use Algorithm \ref{algorithmpn} to produce examples of semistable vector bundles
on more complicated projective varieties.

We omit the famous restriction theorem of Mehta and Ramanathan (see \cite[Theorem 6.1]{mehtaramanathanrestriction} or \cite[Theorem 7.2.8]{huybrechtslehn}), which works over an arbitrary algebraically closed field $K$ and provides the existence of an integer
$k_0$ such that for $c$ general elements $D_1 \komdots D_c \in |\mathcal{O}_{\PP^N}(k)|$ the restriction of a semistable torsion-free sheaf $\shE$ on $\PP^N$ to the smooth complete intersection $D_1 \cap \ldots \cap D_c$ is semistable for all $k \geq k_0$. But their theorem provides no bound for $k_0$. Unlike the theorem of Mehta and Ramanathan, the following restriction theorem of H. Flenner gives an explicit bound for $k_0$, but only works in characteristic $0$.

\begin{theorem}[Flenner]
\label{flennerrestrictiontheorem}
Let $K$ be an algebraically closed field of characteristic $0$ and let $\shE$ be a semistable coherent torsion-free sheaf of rank $r$ on
$\PP^N$. Then for $k$ and $1 \leq c \leq N-1$ fulfilling
$$\frac{{k+N \choose N} -ck-1}{k} > \max \{\frac{r^2-1}{4},1\}$$
and for $c$ general elements $D_1 \komdots D_{c} \in |\mathcal{O}_{\PP^N}(k)|$ the restriction $\shE|_C$ is semi\-stable on the smooth
complete intersection $C=D_1 \cap \ldots \cap D_{c}$.
\end{theorem}

\begin{proof}
See \cite[Theorem 1.2]{flennerrestriction} or \cite[Theorem 7.1.1]{huybrechtslehn}.
\end{proof}

The strongest restriction theorem is due to A. Langer. It works in arbitrary characteristic and gives a degree bound for arbitrary smooth hypersurfaces in projective space. It involves the \emph{discriminant} $\Delta(\shE):=2rc_2(\shE)-(r-1)c_1(\shE)^2$ of a locally free sheaf $\shE$ of rank $r$, where $c_1(\shE)$ and $c_2(\shE)$ denote the first and second Chern class respectively.

\begin{theorem}[Langer]
\label{langerrestrictiontheorem}
Let $K$ be an algebraically closed field and let $\shE$ be a stable coherent torsion-free sheaf of rank $r \geq 2$ on $\PP^N$ and let $D \in |\mathcal{O}_{\PP^N}(k)|$ be a smooth divisor such that $\shE|_D$ is torsion-free. If
$$k > \frac{r-1}{r}\Delta(\shE)+\frac{1}{r(r-1)},$$
then the restriction $\shE|_D$ is stable.
\end{theorem}

\begin{proof}
See \cite[Theorem 2.19]{langersurvey}.
\end{proof}

For a kernel bundle $\shE$ given by a short exact sequence
$$0 \lto \shE \lto \bigoplus_{i=1}^n \mathcal{O}_{\PP^N}(a_i) \lto \bigoplus_{j=1}^m \mathcal{O}_{\PP^N}(b_j) \lto 0$$
we have already mentioned in Section \ref{syzygyandkernelbundles} that $c_1(\shE)= \sum_{i=1}^n a_i -\sum_{j=1}^m b_j$. Since the Chern polynomial is multiplicative on short exact sequences, it is also easy to see that we have
$$c_2(\shE)= \frac{1}{2} \left((\sum_{i=1}^n a_i)^2 - \sum_{i=1}^n a_i^2+ (\sum_{j=1}^m b_j)^2  + \sum_{j=1}^m b_j^2 \right) - \sum_{i,j} a_ib_j$$
and hence
$$\Delta(\shE)=(\sum_{i=1}^n a_i)^2 + (\sum_{j=1}^m b_j)^2 -(n-m)(\sum_{j=1}^m b_j^2 - \sum_{i=1}^n a_i^2) -2 \sum_{i,j}a_ib_j.$$
In particular, we have for a syzygy bundle $\shS:=\Syzfn$ given by homogeneous $R_+$-primary polynomials $f_1 \komdots f_n \in R = K[X_0 \komdots X_N]$ of degrees $d_1 \komdots d_n$ the formula $\Delta(\shS) = (\sum_{i=1}^n d_i)^2 - (n-1) \sum_{i=1}^n d_i^2$.

In \cite{brennerstronglysemistable} H. Brenner has shown that there is no restriction theorem for strong semistability in the sense of
Theorem \ref{langerrestrictiontheorem}. For general hypersurfaces $X \subset \PP^N$ there is the following Flenner-type restriction theorem which is also due to A. Langer.

\begin{theorem}[Langer]
\label{langerstronglysemistablerestriction}
Let $\shE$ be a semistable locally free sheaf of rank $r \geq 2$ on $\PP^N$ over an algebraically closed field $K$ of positive characteristic. Let $k$ be an integer such that
$$k> \frac{1}{2} \max \left\{\Delta(\shE), N^5-2N^3+2N+1  \right \}$$
and
$$\frac{{k+N \choose N} -1}{k} > \max \left \{\frac{r^2-1}{4},1  \right \}+1.$$
Then $\shE|_D$ is strongly semistable on the general hypersurface in $|\mathcal{O}_{\PP^N}(k)|$.
\end{theorem}

\begin{proof}
See \cite[Theorem 3.1]{langersemistablerestrictionnote}.
\end{proof}

\section{Tannaka duality of stable bundles restricted to curves} \label{tannakaoncurves}
We return to the situation of section \ref{tannaka}. Recall that the ground field $K$ is algebraically closed and of characteristic $0$.
Here we investigate the problem of the behavior of Tannaka dual groups after restricting a stable bundle of degree $0$ on $\PP^N$ to  smooth connected curves $X$ such that the restricted bundle is still stable.
The main problem is that the connected Tannaka dual group of a stable vector bundle on $\PP^N$ may become disconnected after restriction to a curve:

\begin{lemma}
Let $X \subset \PP^2$ be a connected smooth curve of genus $>1$. Then there exists a stable bundle of degree $0$ on $\PP^2$ such that its restriction to $X$ is again stable with nonconnected Tannaka dual group. 
\end{lemma}

\begin{proof}Recall that a finite vector bundle is a vector bundle which is trivialized by a finite \'etale morphism.
 There is a one-to-one correspondence between finite vector bundles on $X$ and representations of the \'etale fundamental group $\pi_1(X,x)$ having finite image. Now choose such an irreducible representation of dimension $r$ with trivial determinant, which certainly exists if the genus of the curve is $>1$. It is well-known that the associated vector bundle $E$ is stable of degree $0$, with Tannaka dual group equal to the image of the representation. Further, its determinant is the trivial bundle. Since $E^*(n)$ is generated by $r+1$ global sections (\cite[Lemma 2.3]{brennertestexponent}), it is easy to see that one obtains a short exact sequence $$0 \longrightarrow \mathcal{O}_X(n) \stackrel{\varphi} \longrightarrow \bigoplus_{i = 1}^{r+1}\mathcal{O}_X(d_i) \longrightarrow E^*(n) \longrightarrow 0.$$ The dual of the morphism $\varphi$ lifts to an exact sequence $$0 \longrightarrow \shE \longrightarrow \bigoplus_{i = 1}^{r+1}\mathcal{O}_{\PP^2}(-d_i) \stackrel{\varphi^*} \longrightarrow \mathcal{O}_{\PP^2}(-n)$$ on $\PP^2$. In particular, $\shE|_X = E$ and the singularities of $\shE$ are of codimension $>2$, hence $\shE$ is locally free and of course stable of degree $0$. The Tannaka dual group of $\shE$ is connected by Lemma \ref{connectedlemma}, but the restriction of the bundle $\shE$ to $X$ has a finite Tannaka dual group.
\end{proof}

Hence we need a criterion for the Tannaka dual group to be connected after restricting the bundle $\shE$ to the smooth and connected curve $X$. For the rest of this section we denote by $p$ a prime number and by $\overline{\QQ}_p$ an algebraic closure of the p-adic numbers with ring of integers $\mathfrak{o}$ and residue field $\kappa = \overline{\FF}_p$. We call a finitely presented, flat and proper scheme $\mathfrak{X}$ over $\mathfrak{o}$ together with an isomorphism $X \cong \mathfrak{X} \otimes_{\mathfrak{o}} \overline{\QQ}_p$ a \emph{model} of $X$. Note that any scheme $\mathfrak{X}$ over $\mathfrak{o}$ is the disjoint union of the generic fiber $\mathfrak{X} \otimes \overline{\QQ}_p$, which is open in $\mathfrak{X}$, and the special fiber $\mathfrak{X} \otimes \overline{\FF}_p$, which is closed.  For the rest of this section set $K = \overline{\QQ}_p$.
 
\begin{theorem} \label{connectedtheorem}
 Let $E$ be a vector bundle on the smooth, connected and projective curve $X$ over $\overline{\QQ}_p$. If there exists a model $\mathfrak{X}$ of $X$ together with a vector bundle $\shE$ on $\mathfrak{X}$ such that $E \cong \shE \otimes_{\mathfrak{o}} \overline{\QQ}_p$ and $\shE \otimes_{\mathfrak{o}} \mathfrak{o}/p$ is a trivial bundle on the scheme $\mathfrak{X} \otimes_{\mathfrak{o}} \mathfrak{o}/p$, then $E$ is semistable of degree $0$ with connected Tannaka dual group $G_E$.
\end{theorem}

\begin{proof}
 The proof uses results from non-abelian $p$-adic Hodge theory. The semistability of the bundle is shown in \cite[Theorem 13]{deningerwernerparallel}, the assertion that $G_E$ is connected follows from \cite[Theorem 3.12]{monodromygroups}. 
\end{proof}

\begin{example} \label{exampleofrestriction}
 Let $\shE$ be a vector bundle on $\PP^N$ sitting in the short exact sequence $$0 \longrightarrow \shE \longrightarrow \bigoplus_{i = 1}^{N+1}\mathcal{O}_{\PP^N}(1) \stackrel{\varphi}{\longrightarrow} \mathcal{O}_{\PP^N}(N+1) \longrightarrow 0,$$ where the morphism $\varphi$ is defined by homogeneous polynomials $$f_i := X_0^{i-1}X_1^{N-i+1} + pg_i,\; i = 1,\dots, N+1,$$ with $g_i \in \mathfrak{o}[X_0,\dots,X_N]$. The vector bundle $\shE$ is stable of degree $0$ due to Theorem \ref{bohnhorstspindlerhdoneintro}. If we consider $\shE$ as a sheaf on $\PP^N_{\mathfrak{o}}$, it is easy to see that it is a locally-free sheaf outside the closed subset $\{[0;0;X_2;\dots;X_N]\}\subset \PP^N_{\kappa} \subset \PP^N_{\mathfrak{o}}$.  Then for every smooth connected curve $X \subset \PP^N$ which has a model $\mathfrak{X} \subset \PP^N_{\mathfrak{o}}$ such that the special fiber does not intersect the subspace $\{[0;0;X_2;\dots;X_N]\}$, the vector bundle $\shE|_{\mathfrak{X}}$ has $N$ linearly independent global sections modulo $p$ and hence is trivial. If the curve $X$ is the intersection of $N-1$ smooth divisors of degree $\gg 0$, Theorem \ref{langerrestrictiontheorem} yields that the restriction $E$ of $\shE$ to $X$ is a stable bundle, and it follows from Theorem \ref{connectedtheorem} above that $G_E$ is connected. See also \cite[Example 4.6 and Remark 4.8]{monodromygroups}.
\end{example}

\begin{remark}
 A similar argument was used by H. Brenner to provide examples for stable vector bundles on $p$-adic curves with semistable reduction, see \cite[Remark p. 571]{deningerwernerparallel}.
\end{remark}

In general, we cannot apply Theorem \ref{connectedtheorem} to an arbitrary kernel bundle $\shE$ on $\PP^N$, but in many cases one can apply it to the pull-back $\pi^*(\shE)$ with respect to a suitable chosen finite morphism $\pi: \PP^N \rightarrow \PP^N$ such that $\pi^*(\shE)$ is defined over $\mathfrak{o}$ and such that it is trivial modulo $p$ outside a closed subset of codimension $\geq 2$. Then for all curves $X$  which have a model $\mathfrak{X} \subset \PP^N_{\mathfrak{o}}$ whose special fiber does not intersect this closed subset, we have $G_{\pi^*(\shE)} = G_{\pi^*(\shE)|_X}$ if the curve $X$ is a complete intersection of smooth divisors of sufficiently high degree. We will illustrate this method in the following examples.    

\begin{example}
Consider the syzygy bundles $$ \shE_1 =  \Syz(X^4-Y^4,X^4-Z^4,X^2Y^2,X^2Z^2,Y^2Z^2)(5)$$ $$  \shE_2 = \Syz(X^6\!-Y^4Z^2,Y^6\!-X^2Z^4,X^4Y^2\!-Z^6,X^2Y^4,Y^2Z^4,X^4Z^2,X^2Y^2Z^2)(7)$$  on $\PP^2$ from Example \ref{exampleofstablebundle} and Example \ref{secondexample}. To find a suitable finite morphism $\pi_i: \PP^2 \rightarrow \PP^2$ as explained above, we try to construct a nontrivial morphism $g_i: \PP^1_{\FF_p} \rightarrow \PP^2_{\FF_p}$ such that the  pull-back $g_i^*(\shE_i \otimes \FF_p)$ is the trivial bundle. Then we can choose a rational map $\pi_i: \PP^2_{\ZZ_p} \dashrightarrow \PP^2_{\ZZ_p}$ that is defined outside the point $\{[0;0;1]\} \in \PP^2_{\FF_p} \subset \PP^2_{\ZZ_p}$  such that modulo $p$ there is the commutative diagram $$\xymatrix{\PP^2_{\FF_p}\setminus\{[0;0;1]\} \ar[d] \ar[r]^<<<<<{\pi_{i,\FF_p}} & \PP^2_{\FF_p} \\ \PP^1_{\FF_p} \ar[ur]_{g_i} & }$$ where the vertical morphism is defined as $[X;Y;Z] \mapsto [X;Y]$.  It is then clear that $\pi^*(\shE_i)$ is modulo $p$ the trivial bundle on the open subset $\PP^2_{\FF_p}\setminus \{[0;0;1]\}$. A computation with CoCoA shows that the morphism $g_i$ can for example be chosen as $[X;Y] \mapsto [X;Y;2X+Y]$ if the prime $p$ is odd. Then the rational map $\pi_i$ can be chosen as $[X;Y;Z] \mapsto [X;Y;2X+Y+pZ]$. The restriction to the generic fiber gives the finite morphism $\pi_i: \PP^2 \rightarrow \PP^2$ we were looking for. \medskip

On the generic fiber, the bundles $\pi_i^*(\shE_i)$ are polystable, and one computes $\dim(\mbox{End}(\pi_i^*(\shE_i))) = h^0(\PP^2,\pi_i^*(\shE_i) \otimes \pi_i^*(\shE_i)) = 1$ using Proposition \ref{exteriorpowerpresentations}. It follows that they have to be stable. Let $\mathfrak{X} \subset \PP^2_{\mathfrak{o}}$ be a model of a smooth connected curve $X \subset \PP^2_{\overline{\QQ}_p}$ such that the special fiber $\mathfrak{X}_{\overline{\FF}_p}$  does not contain the point $[0;0;1]$. If the degree of the plane curve $X$ is large enough, we may use Theorem \ref{langerrestrictiontheorem} and assume that the restriction of the bundle $\pi_i^*(\shE_i)$ is still a stable bundle. It follows from  Theorem \ref{connectedtheorem} that its Tannaka dual group is a connected semisimple group. Furthermore, we have $\Gamma(X,\pi_i^*(\shE_i)^{\otimes 4}|_X) = \Gamma(\PP^2_{\overline{\QQ}_p}, f^*(\shE_i)^{\otimes 4})$ for curves of sufficiently large degree. Hence the Tannaka dual groups satisfy $G_{\shE_1|_X} = \mathbf{Sp}_4 \subset \mathbf{GL}_4$ , $G_{\shE_2|_X}=\mathbf{Sp}_6 \subset \mathbf{GL}_6$.

\end{example}

\begin{example}
The syzygy bundle $$\shE = \Syz(X^3,Y^3,Z^3,XYZ)(4)$$ on $\PP^2$ is stable of degree $0$ due to Theorem \ref{monomialcase}, with Tannaka dual group $G_{\shE} = \mathbf{SL}_3 \subset \mathbf{GL}_3$ because of $\dim(\Gamma(\PP^2, \shE^{\otimes 3})) = 1$. The morphism $g: \PP^2_{\FF_p}\rightarrow \PP^2$ can be chosen as $[X;Y] \mapsto [X^2+Y^2;X^2;X^2+XY]$ and hence
 $\pi: \PP^2_{\ZZ_p} \dashrightarrow \PP^2_{\ZZ_p}$ for example as $[X;Y;Z] \mapsto [X^2+Y^2;X^2;X^2+XY + pZ^2]$. The same computations and arguments as above then show that $G_{\shE|_X} = \mathbf{SL}_3$ for a plane curve $X$ of sufficiently large degree as in the above example. 

\end{example}

It is natural to ask if this method works for all semistable vector bundles of degree $0$ on $\PP^N$:
\begin{question}
 Let $\shE$ be a semistable vector bundle of degree $0$ on $\PP^N$. Is there always a finite morphism $\pi: \PP^N \rightarrow \PP^N$ and a model $\mathfrak{P}$ of $\PP^N$ such that $\pi^*(\shE)$ lifts to a sheaf on $\mathfrak{P}$ which is modulo $p$ a free sheaf outside a closed subset of codimension $\geq 2$?  
\end{question}

\section{The stability of the syzygy bundle of five generic quadrics}
\label{stabilityofgenericquadrics}

It is an open question of whether for generic forms $f_1 \komdots f_n$ of degrees $d_1 \komdots d_n$ in the polynomial ring $R=K[X_0 \komdots X_N]$ over an algebraically closed field $K$ the corresponding syzygy bundle is semistable or even stable. There is no chance if the $d_i$'s do not satisfy the necessary degree condition of Proposition \ref{linebundlequotients}. Hence, the question only makes sense if the necessary condition on the degrees is fulfilled, e.g., if we consider forms of constant degree. Since semistability is an open property, it is enough to find a single $R_+$-primary family $g_1 \komdots g_n$ having the same degree configuration such that $\Syz(g_1 \komdots g_n)$ is semistable.

Via $R_+$-primary monomial families $f_i=X^{\sigma_i}$, $d_i=|\sigma_i|$, one can use Brenner's result \ref{monomialcase} to establish generic semistability in a combinatorial way by producing examples of monomial families with semistable syzygy bundle. This has been done recently in \cite[Theorem 4.6]{marquesmiroroig}, where P. Macias Marques and R. M. Mir\'{o}-Roig have proved the stability of the syzygy bundle $\Syzfn$ on $\PP^N$ for generic forms of degree $d$ with $N+1 \leq n \leq {d+N \choose N}$, $(N,d,n)\neq (2,2,5)$. This extends \cite[Theorem 3.5]{costamiroroigmarques} of L. Costa, R. M. Mir\'{o}-Roig and P. Macias Marques, where only the case $N=2$ has been proven. The general result of \cite{marquesmiroroig} has also been obtained simultaneously by I. Coand\u{a} in \cite{coandasyzygybundles}. For the case $N=2$, $n=5$ and $d=2$, where only semistability has been shown, Macias Marques asks the following question (see \cite[Problem 2.9]{marquesthesis}).

\begin{problem}[Macias Marques]
\label{questionmarques}
Is there a family of five quadratic homogeneous polynomials in $K[X_0,X_1,X_2]$ such that their syzygy bundle is stable?
\end{problem}

Note that one cannot establish generic semistability via a monomial example since for the only candidate we have
$$\mu(\Syz(X^2,Y^2,Z^2,XY,XZ)) = -\frac{5}{2}=\frac{1-6}{2} = \mu(\Syz(X^2,XY,XZ)).$$
We answer Macias Marques's question in the following proposition.

\begin{proposition}
\label{answertomarquesproblem}
The syzygy bundle $$\Syz(X^2-Y^2,X^2-Z^2,XY,XZ,YZ)$$ is stable on $\PP^2=\Proj K[X,Y,Z]$. Moreover, the syzygy bundle for five generic quadrics in $K[X,Y,Z]$ is stable on the projective plane.
\end{proposition}

\begin{proof}
The syzygy bundle $\shS = \Syz(X^4-Y^4,X^4-Z^4,X^2Y^2,X^2Z^2,Y^2Z^2)$, which we have considered in Example \ref{exampleofstablebundle}, is the pull-back of $\shE:=\Syz(X^2-Y^2,X^2-Z^2,XY,XZ,YZ)$ under the finite morphism
$$\PP^2 \lto \PP^2,~X \longmapsto X^2,~ Y \longmapsto Y^2,~ Z \longmapsto Z^2.$$ Since $\shS$ is a stable bundle, so is $\shE$. The supplement follows from the openness of stability.
\end{proof}

\begin{remark}
In the recent preprint \cite{coandasyzygybundles} I. Coand\u{a} has independently proved in [ibid., Example 1.3] the stability of the generic syzygy bundle for $(N,d,n)=(2,2,5)$. But his proof is more complicated and does not provide an explicit example of a family of five homogeneous quadrics in three variables.
\end{remark}

Let $\shM_{\PP^N}(n-1,c_1,\dots, c_N)$ be the moduli space of stable vector bundles of rank $n-1$ and Chern classes $c_1,\dots, c_N$ on the projective space $\PP^N$. Denote by $\shS_{(N,n,d)} \subset \shM_{\PP^N}(n-1,c_1,\dots, c_N)$ the stratum of stable syzygy bundles $\shE$ defined by the short exact sequence $$0 \lto \shE \lto \bigoplus_{i=1}^n \mathcal{O}_{\PP^N}(-d) \lto \mathcal{O}_{\PP^N} \lto 0.$$ In his thesis \cite{marquesthesis}, M. Marques computes the dimension of the syzygy stratum and the codimension of its closure in the irreducible component of the moduli space. However, he could not give an answer for the case $N = 2, d = 2, n = 5$ due to the lack of a stable syzygy bundle of five homogeneous quadrics. Our example also closes this gap. 
We recall that the moduli spaces of stable bundles on $\PP^2$ with fixed invariants are irreducible and their dimensions are known, see for example \cite{moduliirreducible}.
 
\begin{corollary}
The syzygy stratum $\shS_{(2,5,2)} \subset \shM_{\PP^2}(4,-10,40)$ has dimension $5$. In particular, $\overline{\shS}_{(2,5,2)} = \shM_{\PP^2}(4,-10,40)$.
\end{corollary}

\begin{proof}
See \cite[Proposition 4.2 and Theorem 4.3]{marquesthesis}, where the proof presented there works analogously for the case
$N=2$, $n=5$, $d=2$ due to Proposition \ref{answertomarquesproblem}.
\end{proof}

\section{Computing inclusion bounds for tight closure and solid closure}
\label{computingtightclosure}

Our semistability algorithm also has impact on certain ideal closure operations in commutative algebra due to a geometric interpretation by H. Brenner. 
We recall briefly the notions of \emph{tight closure} and \emph{solid closure}, where we restrict ourselves
to the case where the ring $R$ under consideration is a Noetherian integral domain. For a detailed exposition of these closure operations and their backround see \cite{hunekeapplication}. 

Let $I=(f_1 \komdots f_n) \subseteq R$ be an ideal and $f \in R$. The $R$-algebra $A=R[T_1 \komdots T_n]/(f_1 T_1 + \ldots + f_n T_n+f)$ is called the \emph{forcing algebra} for the elements $f_1 \komdots f_n,f \in R$. The element $f$
belongs to the \emph{solid closure}, which we denote by $I^\star$, if and only if the following holds: For every maximal ideal $\fom$ of $R$ the top-dimensional local cohomology module $H_{\fom^\prime}^d(A^\prime)$ does not vanish, where $A^\prime$ is the forcing algebra for the given data over the local complete domain $R^\prime:=\hat{R}_\fom$ and $d=\dim(R^\prime)=\height(\fom)$.

Now assume that $R$ is of positive characteristic $p$ (i.e., $R$ contains a field of positive characteristic). Then the \emph{tight closure} of $I$ is defined as the ideal
$$I^*:=\{f \in R: \mbox{ there exists } 0 \neq t \in R \mbox{ such that } tf^q \in I^{[q]} \mbox{ for all } q=p^e\},$$
where $I^{[q]} = (f_1^q \komdots f_n^q)$ denotes the extended ideal under the $e$-th iteration of the Frobenius $F: R \ra R$, $f \mapsto f^p$.

An important fact due to M. Hochster is that $I^*=I^\star$ holds in positive characteristic for a normal $K$-algebra $R$ of finite type (in fact this is true under
weaker assumptions); see \cite[Theorem 8.6]{hochstersolid}.

It follows already from the definitions of these closure operations that they are hard to compute. For a normal standard graded integral $2$-dimensional algebra $R$ over an algebraically closed field $K$ there is a well-developed theory by H. Brenner for solid closure and tight closure which connects these notions with (strong) semistability of the corresponding syzygy bundle $\Syz(f_1 \komdots f_n)$ on the smooth projective curve $C = \Proj R$; see \cite{brennerbarcelona} for an excellent survey. This geometric approach combined with the  restriction theorems presented in Section \ref{sectionrestrictiontheorems} enables us to use our semistability Algorithm \ref{algorithmpn} to compute inclusion bounds for solid closure and tight closure in homogeneous coordinate rings of smooth projective curves, particularly plane curves, of sufficiently large degree.

In characteristic $0$ there is the following result for solid closure.

\begin{theorem}[Brenner]
\label{inclusionsolidclosure}
Let $K$ be an algebraically closed field of characteristic zero and $R$ be a normal standard graded $K$-domain of dimension two. Further let $I=(f_1 \komdots f_n)$ be an $R_+$-primary homogeneous ideal. If $\Syz(f_1 \komdots f_n)$ is semistable on $C = \Proj R$ then $$I^\star = I + R_{\frac{d_1 + \ldots + d_n}{n-1}},$$ where $d_i = \deg(f_i)$ for $i = 1 \komdots n$.
\end{theorem}

\begin{proof}
See the characteristic zero version of \cite[Theorem 6.4]{brennerbarcelona}.
\end{proof}

So Theorem \ref{inclusionsolidclosure} gives for an element $f \in R_m$ an inclusion  $f \in I^\star$ for $m \geq  \frac{d_1 + \ldots + d_n}{n-1}$ and for $m < \frac{d_1 + \ldots + d_n}{n-1}$ the question of whether $f$ belongs to $I^\star$ reduces to an ideal membership test which is a well-known 
procedure in computational algebra (cf. for instance \cite[Proposition 2.4.10]{kreuzerrobbiano}).

The following theorem works in positive characteristic under the assumption that the syzygy bundle is strongly semistable. We recall that a vector bundle $\shE$ is \emph{strongly semistable} if for every $e \geq 0$ the Frobenius pull-backs ${F^e}^*(\shE)$ are semistable. So our
algorithmic methods for tight closure can either be applied to homogeneous coordinate rings of the general curve of large degree (via Theorem \ref{langerstronglysemistablerestriction}) or to coordinate rings of elliptic curves (there semistable bundles are strongly semistable by \cite[Theorem 2.1]{mehtaramanathanhomogeneous}).

\begin{theorem}[Brenner]
\label{inclusiontightclosure}
Let $K$ be an algebraically closed field of characteristic $p>0$ and $R$ be a normal standard graded $K$-domain of dimension two. Further let $I=(f_1 \komdots f_n)$ be an $R_+$-primary homogeneous ideal such that $\Syz(f_1 \komdots f_n)$ is strongly semistable
on $C = \Proj R$. Denote the genus of $C$ by $g$. Then the following hold.
\begin{enumerate}
\item If $m \geq \frac{d_1 + \ldots + d_n}{n-1}$ then $R_m \subseteq I^*$.
\item If $m < \frac{d_1 + \ldots + d_n}{n-1}$ and $f \in R_m$ then $f \in I^*$ if and only if
\begin{enumerate}
\item $f^p \in I^{[p]}=(f_1^p \komdots f_n^p)$ if $p > 4(g-1)(n-1)^3$
\item or $f^q \in  I^{[q]}=(f_1^q \komdots f_n^q)$ for $q=p^e>6g$ if $p < 4(g-1)(n-1)^3$.
\end{enumerate}
\end{enumerate}
\end{theorem}   

\begin{proof}
See the positive characteristic version of \cite[Theorem 6.4]{brennerbarcelona}.
\end{proof}

\begin{remark}
Let $\shC=\Proj R \ra \Spec \ZZ$ be a generically smooth projective relative curve and $I:=(f_1 \komdots f_n)$ be an $R_+$-primary ideal. In this situation one can deduce tight closure information of the reductions $I_p$ in the fiber rings $R_p:=R \otimes_{\ZZ} \FF_p$ from semistability in characteristic $0$. Let $\shS:=\Syz(f_1 \komdots f_n)$ denote the syzygy bundle on the total space $\shC$. If $\shS_0:=\shS|_{\shC_0}$ is semistable on the generic fiber $\shC_0:=\shC \times_{\Spec \ZZ} \Spec \QQ$
and $m > \frac{d_1 + \ldots + d_n}{n-1}$ then $\shS_0(m)$ has positive degree and is therefore ample (see \cite[Theorem 2.4]{hartshorneamplecurve}). Since ampleness is an open property, the reductions to positive characteristic $\shS_p:=\shS|_{\shC_p}$ and $(\shS_p(m))^*$ are also ample on the special fibers $\shC_p:=\shC \times_{\Spec \ZZ} \Spec \FF_p$ for almost all prime numbers $p \in \ZZ$. In this situation Brenner's geometric approach also yields results on the tight closure of $I_p \subseteq R_p$ for $p \gg 0$; see \cite[Section 4]{brennerbarcelona} for a detailed treatment of ampleness criteria for tight closure. In particular, by [ibid., Proposition 4.17] we have $(R_p)_m \subseteq I_p^*$ (in fact $(R_p)_m$ already belongs to the \emph{Frobenius closure} $I_p^F:=\{f \in R_p: f^q \in I_p^{[q]} \mbox{ for some } q=p^e\} \subseteq I_p^*$).
\end{remark}

\begin{remark}
What can be said in higher dimensions? As usual, we consider $R_+$-primary homogeneous polynomials $f_1 \komdots f_n$ in $P=K[X_0 \komdots X_N]$. We can compute a minimal graded free resolution $\mathfrak{F}_\bullet$ of the ideal $(f_1 \komdots f_n)$; see \cite[Section 4.8.B]{kreuzerrobbiano2} for the computational backround. Since the quotient $R=P/I$ is Artinian, the length of $\mathfrak{F}_\bullet$ equals $N+1$ by the Auslander-Buchsbaum formula. Consequently, the corresponding resolution of the associated sheaves on $\PP^N$ gives a resolution
$$\mathfrak{F}_\bullet: 0 \lto \shF_{N+1} \lto \shF_{N} \lto \ldots \lto \shF_1 \lto \mathcal{O}_{\PP^N} \lto 0$$
of the structure sheaf with splitting bundles $\shF_i$, $i=1 \komdots N+1$. Instead of looking at $\Syzfn = \ker (\shF_1 \ra \mathcal{O}_{\PP^N})$, we consider the bundle $$\Syz_{N-1}:=\Syz_{N-1}(f_1\komdots f_n) := \im(\shF_{N+1} \lto \shF_{N}).$$
To check whether $\Syz_{N-1}$ is semistable, we can apply Algorithm \ref{algorithmpn} to its dual $(\Syz_{N-1})^*$ which is a kernel bundle. If the answer is positive, then we obtain an inclusion bound for the tight closure $(f_1 \komdots f_n)^*$ in the homogeneous coordinate ring $R$ of a generic hyperplane $X \subset \PP^N$ of sufficiently large degree. This works as follows. If we restrict the resolution $\mathfrak{F}_\bullet$ to $X$, we obtain an exact complex of sheaves on $X$ and the sheaf $\Syz_{N-1}|_X$ is strongly semistable for $k= \deg(X) \gg 0$ by Theorem \ref{langerstronglysemistablerestriction}. Then Brenner's result \cite[Theorem 2.4]{brennerlinearfrobenius} gives the inclusion bound $R_{\geq \nu} \subseteq (f_1 \komdots f_n)^*$, where $\nu:=-\frac{\mu(\Syz_{N-1}|_X)}{\deg(X)}$. Note that we can compute all necessary invariants (rank, degree, discriminant) of $\Syz_{N-1}$ from the resolution $\mathfrak{F}_\bullet$.

We obtain the same inclusion bound if we restrict $\Syz_{N-1}$ to smooth hypersurfaces $X \subset \PP^N$ for which every semistable bundle on $X$ is strongly semistable; compare for instance Lemma \ref{semistabilitytensorposchar} and Example \ref{calabiyauexample}. Here the degree bound for $\deg(X)$, which ensures the semistability of $\Syz_{N-1}|_X$, is given by Theorem \ref{langerrestrictiontheorem}.
\end{remark}

\bibliographystyle{amsplain}

\begin{thebibliography}{10}

\bibitem{douhcompact}
V.~Balaji, \emph{Principal bundles on projective varieties and the
  {D}onaldson-{U}hlenbeck compactification}, J. Differential Geom. \textbf{76}
  (2007), no.~3, 351--398. 

\bibitem{adddouhcompact}
\bysame, \emph{Addendum to ``{P}rincipal bundles on projective varieties and
  the {D}onaldson-{U}hlenbeck compactification''}, J. Differential
  Geom. \textbf{83} (2009), no.~2, 461--463. 

\bibitem{bohnhorstspindler}
G.~Bohnhorst and H.~Spindler, \emph{The stability of certain vector bundles on
  {${\PP}^n$}}, Complex algebraic varieties, Lect. Notes Math., vol. 1507,
  1992, pp.~39--50.

\bibitem{brennerlinearfrobenius}
H.~Brenner, \emph{A linear bound for {F}robenius powers and an inclusion bound
  for tight closure}, Mich. Math. J. \textbf{53} (2005), no.~3, 585--596.

\bibitem{brennerstronglysemistable}
\bysame, \emph{There is no {B}ogomolov type restriction theorem for strong
  semistability in positive characteristic}, Proc. Amer. Math. Soc.
  \textbf{133} (2005), no.~7, 1941--1947.

\bibitem{brennertestexponent}
\bysame, \emph{Bounds for test exponents}, Compositio Math. \textbf{142}
  (2006), 451--463.

\bibitem{brennerlookingstable}
\bysame, \emph{Looking out for stable syzygy bundles}, Adv. Math. \textbf{219}
  (2008), no.~2, 401--427.

\bibitem{brennerbarcelona}
\bysame, \emph{Tight closure and vector bundles}, Three Lectures on Commutative
  Algebra (J.~Elias S.~Zarzuela G.~Colom\'{e}-Nin, T. Cortodellas~Benitez,
  ed.), University Lecture Series, vol.~42, AMS, 2008, pp.~1--71.

\bibitem{buchsbaumeisenbud}
D.~Buchsbaum and D.~Eisenbud, \emph{Algebra structures for finite free
  resolutions, and some structure theorems for ideals of codimension $3$},
  Amer. J. of Math. \textbf{99} (1977), 447--485.

\bibitem{CartanEilenberg}
H.~Cartan and S.~Eilenberg, \emph{Homological algebra}, Princeton Landmarks in
  Mathematics, Princeton University Press, Princeton, NJ, 1999, With an
  appendix by David A. Buchsbaum, Reprint of the 1956 original.

\bibitem{coandasyzygybundles}
I.~Coand\u{a}, \emph{On the stability of syzygy bundles}, Preprint (2009),
  arXiv:0909.4435.

\bibitem{CocoaSystem}
{CoCoA}Team, \emph{{{\hbox{\rm C\kern-.13em o\kern-.07em C\kern-.13em
  o\kern-.15em A}}}: a system for doing {C}omputations in {C}ommutative
  {A}lgebra}, Available at \/ {\tt http://cocoa.dima.unige.it}.

\bibitem{costamiroroigmarques}
L.~Costa, P.~Macias Marques, and R.~M. Mir\'{o}-Roig, \emph{Stability and
  unobstructedness of syzygy bundles}, J. Pure Appl. Algebra \textbf{214} (2010), 1241--1262.

\bibitem{delignemilne}
P.~Deligne, J.~S. Milne, A.~Ogus, and K.~y.~Shih, \emph{Hodge cycles, motives,
  and {S}himura varieties}, Lecture Notes in Mathematics, vol. 900,
  Springer-Verlag, Berlin, 1982.

\bibitem{deningerwernerparallel}
C.~Deninger and A.~Werner, \emph{Vector bundles on $p$-adic curves and parallel
  transport}, Ann. Scient. \'{E}c. Norm. Sup. \textbf{38} (2005), 535--597.

\bibitem{eisenbud}
D.~Eisenbud, \emph{Commutative algebra with a view toward algebraic geometry},
  Springer-Verlag, 1995.

\bibitem{moduliirreducible}
G.~Ellingsrud, \emph{Sur l'irr\'eductibilit\'e du module des fibr\'es stables
  sur {${\bf P}^{2}$}}, Math. Z. \textbf{182} (1983), no.~2, 189--192.
  

\bibitem{flennerrestriction}
H.~Flenner, \emph{Restrictions of semistable bundles on projective varieties},
  Comment. Math. Helv. \textbf{59} (1984), 635--650.

\bibitem{LiE}
Centre for Mathematics and Computer~Science in~Amsterdam, \emph{Lie: A computer
  algebra package for lie group computations}, Available at \:\:\:\:\:\:\:\:\:\:\:\:\:\:\:\:\:\;\;\;\:\:\:\:\:\:\:\:\:\;\;\:\:\:\:\:\:\:\:\:\:\;\;\;\:\:\/ {\tt
  http://young.sp2mi.univ-poitiers.fr/~marc/LiE/}.

\bibitem{hartshorneamplesubvarieties}
R.~Hartshorne, \emph{Ample subvarieties of algebraic varieties}, Springer,
  Berlin Heidelberg New York, 1970.

\bibitem{hartshorneamplecurve}
\bysame, \emph{Ample vector bundles on curves}, Nagoya Math. J. \textbf{43}
  (1971), 73--89.

\bibitem{hartshornealgebraic}
\bysame, \emph{Algebraic {G}eometry}, Springer, New York, 1977.

\bibitem{hartshornestablereflexive}
\bysame, \emph{Stable reflexive sheaves}, Math. Ann. \textbf{254} (1980),
  121--176.

\bibitem{hochstersolid}
M.~Hochster, \emph{Solid closure}, Contemp. Math. \textbf{159} (1994),
  103--172.

\bibitem{hoppe}
H.J. Hoppe, \emph{{G}enerischer {S}paltungstyp und zweite {C}hernklasse
  stabiler {V}ektorraum-b\"undel vom {R}ang $4$ auf ${\PP}^4$}, Math. Z.
  \textbf{187} (1984), 345--360.

\bibitem{hunekeapplication}
C.~Huneke, \emph{Tight closure and its applications}, CBMS Lecture Notes in
  Mathematics, vol.~88, AMS, Providence, 1996.

\bibitem{huybrechtslehn}
D.~Huybrechts and M.~Lehn, \emph{The {G}eometry of {M}oduli {S}paces of
  {S}heaves}, Viehweg, 1997.

\bibitem{gorensteinalgebras}
A.~Iarrobino and V.~Kanev, \emph{Power sums, {G}orenstein algebras, and
  determinantal loci}, Lecture Notes in Mathematics, vol. 1721,
  Springer-Verlag, Berlin, 1999, Appendix C by Iarrobino and Steven L. Kleiman.

\bibitem{kaiddissertation}
A.~Kaid, \emph{{On semistable and strongly semistable syzygy bundles}},
  PhD-thesis, University of Sheffield, 2009.

\bibitem{monodromygroups}
R.~Kasprowitz, \emph{Monodromy groups of vector bundles on p-adic curves},
  Preprint (2010), arXiv: 1005.5266.

\bibitem{kreuzerrobbiano}
M.~Kreuzer and L.~Robbiano, \emph{Computational {C}ommutative {A}lgebra},
  vol.~1, Springer, Berlin, 2000.

\bibitem{kreuzerrobbiano2}
\bysame, \emph{Computational {C}ommutative {A}lgebra}, vol.~2, Springer,
  Berlin, 2005.

\bibitem{langersurvey}
A.~Langer, \emph{Moduli spaces of sheaves and principal ${G}$-bundles},
  Algebraic geometry, Seattle 2005 (D.~Abramovich et~al., ed.), Proc. Symp.
  Pure Math., vol.~80, AMS, 2009, pp.~273--308.

\bibitem{langersemistablerestrictionnote}
\bysame, \emph{A note on restriction theorems for semistable sheaves}, Math. Res. Lett.
\textbf{17} (2010), no. 05, 823-832.

\bibitem{marquesthesis}
P.~Macias Marques, \emph{Stability and moduli spaces of syzygy bundles}, Tesi
  de doctorat, Universitat de Barcelona, 2009, arXiv:0909.4646.

\bibitem{marquesmiroroig}
P.~Macias Marques, and R.~M. Mir\'{o}-Roig, \emph{Stability of syzygy bundles}, 
Proc. Amer. Math. Soc., posted on January 28, 2011, PII S 0002-9939(2011)10745-7  (to appear in print).

\bibitem{tableslie}
W.~G. McKay, J.~Patera, and D.~W. Rand, \emph{Tables of representations of
  simple {L}ie algebras. {V}ol. {I}}, Universit\'e de Montr\'eal Centre de
  Recherches Math\'ematiques, Montreal, QC, 1990, Exceptional simple Lie
  algebras, With a foreword by R. V. Moody.

\bibitem{mehtaramanathanhomogeneous}
V.~B. Mehta and A.~Ramanathan, \emph{Homogeneous bundles in characteristic
  $p$}, Algebraic Geometry - open problems, Lect. Notes Math., vol. 997, 1982,
  pp.~315--320.

\bibitem{mehtaramanathanrestriction}
\bysame, \emph{Semistable sheaves on projective varieties and the restrictions
  to curves}, Math. Ann. \textbf{258} (1982), 213--226.

\bibitem{nori}
M.~V. Nori, \emph{On the representations of the fundamental group}, Compositio
  Math. \textbf{33} (1976), no.~1, 29--41.

\bibitem{okonekschneiderspindler}
C.~Okonek, M.~Schneider, and H.~Spindler, \emph{{V}ector {B}undles on {C}omplex
  {P}rojective {S}paces}, Birkh\"auser, 1980.

\bibitem{ottavianinotes}
G.~Ottaviani and J.~Valles, \emph{{Moduli of vector bundles and group action}},
  Notes, 2006, available at \url{http://www.dmi.unict.it/~ragusa/}.

\bibitem{peternellsubsheaves}
T.~Peternell, \emph{Subsheaves in the tangent bundle: {I}ntegrability,
  stability and positivity}, School on vanishing theorems and effective results
  in algebraic geometry (J.~P.~Demailly et~al., ed.), ICTP Lect. Notes.,
  vol.~6, The Abdus Salam International Centre for Theoretical Physics, 2001,
  pp.~285--334.

\bibitem{ramananramanathan}
S.~Ramanan and A.~Ramanathan, \emph{Some remarks on the instability flag},
  Tohoku Math. J. \textbf{36} (1984), 269--291.

\bibitem{schejastorch2}
G.~Scheja and U.~Storch, \emph{Lehrbuch der {A}lgebra}, vol.~2, Teubner,
  Stuttgart, 1988.

\end{thebibliography}
\providecommand{\bysame}{\leavevmode\hbox to3em{\hrulefill}\thinspace}
\providecommand{\MR}{\relax\ifhmode\unskip\space\fi MR }
\providecommand{\MRhref}[2]{%
  \href{http://www.ams.org/mathscinet-getitem?mr=#1}{#2}
}
\providecommand{\href}[2]{#2}

\end{document}